\documentclass[12pt,fleqn]{article}

\usepackage{amsmath}
\usepackage{amsthm}
\usepackage{amssymb}
\usepackage{amsfonts}
\usepackage{epsf}
\usepackage{graphicx}
\usepackage{hyperref}
\usepackage{color}
\usepackage{bm}
\usepackage{booktabs}
\usepackage{mathrsfs}

\newtheorem{theorem}{Theorem}

\newtheorem{proposition}{Proposition}
\theoremstyle{remark}

\newtheorem{remark}{Remark}
\theoremstyle{definition}

\newtheorem{definition}{Definition}
\DeclareMathOperator*\Res{{Res}}
\DeclareMathOperator\re{{Re}}
\DeclareMathOperator\im{{Im}} \numberwithin{equation}{section}
\DeclareMathOperator\Ai{{Ai}}
\DeclareMathOperator\Bi{{Bi}}
\DeclareMathOperator\cL{{cL}}
\DeclareMathOperator\dL{{dL}}

\newcommand{\D}{\displaystyle}

\numberwithin{equation}{section}

\newcounter{comment}
\begin{document}

\title{Uniform asymptotics for the discrete Laguerre polynomials}
\date{\today}
\author{Dan Dai$^{\ast}$ and Luming Yao$^{\dag}$}

\maketitle
\begin{abstract}
    In this paper, we consider the discrete Laguerre polynomials $P_{n, N}(z)$ orthogonal with respect to the weight function $w(x) = x^{\alpha} e^{-N cx}$ supported  on the infinite nodes $L_N = \{ x_{k,N} = \frac{k^2}{N^2}, k \in \mathbb{N} \}$. We focus on the ``band-saturated region" situation when the parameter $c > \frac{\pi^2}{4}$.  As $n \to \infty$, uniform expansions for $P_{n, n}(z)$ are achieved for $z$ in different regions in the complex plane. Typically, the Airy-function expansions and Gamma-function expansions are derived for $z$ near the endpoints of the band and the origin, respectively.  The asymptotics for the normalizing coefficient $h_{n, N}$, recurrence coefficients $\mathscr{B}_{n, N}$ and $\mathscr{A}_{n, N}^2$, are also obtained. Our method is based on the Deift-Zhou steepest descent method for Riemann-Hilbert problems.
\end{abstract}

\vspace{6cm}



\vspace{5mm}

\hrule width 65mm

\vspace{2mm}

\begin{description}

\item \hspace*{5mm}$\ast$ Department of Mathematics, City University of
Hong Kong, Hong Kong. \\
Email: \texttt{dandai@cityu.edu.hk}

\item \hspace*{5mm}$\dag$ Department of Mathematics, City University of
Hong Kong, Hong Kong. \\
Email: \texttt{lumingyao2-c@my.cityu.edu.hk}

\end{description}

\section{Introduction}
\label{sec:dl:intro}

The well-known  Laguerre polynomials $L_n^{(\alpha)}(x)$ are given by
\begin{equation}
L^{(\alpha)}_{n}\left(x\right)=\frac{{\left(\alpha+1\right)_{
n}}}{n!}{{}_{1}F_{1}}\left({-n\atop\alpha+1};x\right) , \qquad \alpha > -1,
\end{equation}
where ${}_1F_1$ is the generalized hypergeometric function. They are orthogonal with respect  to the following weight function
\begin{equation}
w(x) = x^\alpha e^{-x}, \qquad x \in (0, \infty).
\end{equation}
In this paper, we will study their discrete analogues. It is well-known that both continuous and discrete orthogonal polynomials play a significant role in various fields of mathematical physics, such as random matrix theory, random tiling, quantum mechanics, etc. For example, one may refer to \cite{baik2007,Ble:Lie2014,Joh2001,Joh2002,karl2012,Lie:Wang2017} for  applications of discrete orthogonal polynomials.

In the study of orthogonal polynomials, one of central problems is to investigate their asymptotic properties when the polynomial degree becomes large. For discrete orthogonal polynomials, a breakthrough is made by Baik et al. \cite{baik2007}, who successfully applied the Riemann-Hilbert (RH) method to derive uniform asymptotics for a general class of weight functions. Their method is also an important development of the original Deift-Zhou method to solve asymptotic problems for the continuous orthogonal polynomials, integrable partial differential equations and random matrix theory; see \cite{deift1999-1,deift1999-2,deift1993}. Since the work of Baik et al. \cite{baik2007}, there has been a lot of work in the study of asymptotics of discrete orthogonal polynomials. For example, Wong and his co-workers derived global asymptotics of various polynomials, with finite lattice \cite{dai2007,lin2013} and infinite lattice \cite{ou2010,Wang:Wong2011}. In the treatment of the so-called ``band-saturated region" endpoints, Bleher and Liechty \cite{bleher2011,Ble:Lie2014} made a major modification to the method in \cite{baik2007}. Wu et al. \cite{Wu:Lin:Xu:Zhao} studied the case where the orthogonality lattice consists of infinite nodes with an accumulation point.  Except \cite{Wu:Lin:Xu:Zhao}, the orthogonality lattices are always equally spaced. In the present paper, we will consider another example where the lattice is not uniformly distributed. By using the Deift-Zhou steepest descent method for RH problems, we  derive the uniform asymptotics for discrete Laguerre polynomials.



The weight function for the discrete Laguerre polynomials is given by
\begin{equation}\label{weight}
    w_N(x) = x^{\alpha} e^{-N cx}, \qquad \qquad \alpha > -1, N \in \mathbb{N} \textrm{ and } c > 0,
\end{equation}
which is supported on the discrete infinite lattice
\begin{equation} \label{ortho-lattice}
    L_N := \Big\{ x_{k,N} = \frac{k^2}{N^2}, k \in \mathbb{N} \Big\} .
\end{equation}
That is, let $P_{n, N}(x)$ be the monic discrete Laguerre polynomials
\begin{equation}\label{Pn,n}
    P_{n, N}(x) = x^n + p_{n, n-1} x^{n-1} + \dots + p_{n, 0}, \quad n = 0, 1, \dots,
\end{equation}
they satisfy the following orthogonality condition
\begin{equation}\label{orthogonality}
    \sum_{x \in L_{N}} P_{m, N} (x) P_{n, N} (x) w_N(x) = h_{n, N} \delta_{mn},
\end{equation}
with $h_{n, N}$ being the normalizing constants and $\delta_{mn} = \begin{cases}
    1, & m = n\\
    0, & m \neq n
\end{cases}$.
The polynomials also satisfy a three-term recurrence relation as follows:
\begin{equation}\label{recurrence}
    x P_{n, N}(x) = P_{n+1, N}(x) + \mathscr{B}_{n, N} P_{n, N} (x) + \mathscr{A}_{n, N}^2 P_{n-1, N}(x).
\end{equation}

Comparing with continuous cases, there is a significant difference in the zero distribution for the discrete orthogonal polynomials:  their zeros are confined by the orthogonality nodes $x_{k,N}$ in $L_N$. More precisely, we have the following proposition.

\begin{proposition}\label{zerosdistr}
    The polynomial $P_{n, N}(x)$ has $n$ real simple zeros and no more than one zero lies in the closed interval $[x_{k,N}, x_{k+1, N}]$ between any two consecutive nodes.
\end{proposition}

\begin{proof}
    This property is an easy consequence of the orthogonality condition \eqref{orthogonality}; see \cite{baik2007,Kui:Rak1998} for detailed proofs.
\end{proof}

The above Proposition implies that there is an upper constraint for the limiting zero distribution of discrete orthogonal polynomials, which depends on the distribution of the orthogonality nodes $x_{k,N}$. It is a well-known fact that the limiting zero distribution of orthogonal polynomials is related to the following constrained equilibrium problem (cf.  \cite{dragnev1997}): given a function $V(x)$ and a Lebesgue measure $\sigma$, find a probability measure $\mu_0 \in \mathcal{M}$ on $\mathbb{R}$ with
\begin{equation} \label{set-M}
\mathcal{M} = \{ 0 \leq \mu \leq \sigma, \ \mu(\mathbb{R}) = 1 \},
\end{equation}
such that $\mu_0$ minimizes the energy functional
\begin{equation} \label{em-def}
  E(\mu_0) =  \inf_{\mu \in \mathcal{M}} E(\mu) = \inf_{\mu \in \mathcal{M}} \left\{ \iint \log \frac{1}{|x-y|} d\mu(x)  d\mu(y) + \int V(x) d\mu(x) \right\}.
\end{equation}
In the above formula, $V(x)$ is also called an external field and  the minimizer $\mu_0$ is the corresponding \emph{equilibrium measure}. For the polynomials with the weight function \eqref{weight}, we have $V(x) = cx$ and $\mu_0$ is the limiting zero distribution of $P_{n, n}(z)$ as $n \to \infty$. Note that, in the case for continuous orthogonal polynomials, the energy functional is minimized among all probability measure on $\mathbb{R}$, where the upper constraint $\sigma$ in \eqref{set-M} does not exist. For the discrete orthogonal polynomials, due to the Proposition \ref{zerosdistr}, the upper constraint $\sigma$ occurs, which describes the distribution of the orthogonality nodes.

Depending on properties of the equilibrium measure $\mu_0$, the real line can be divided into subintervals of the following three different types; see \cite{baik2007}.
\begin{definition}\label{void}
    A \emph{void} is an open interval with $ \mu_0 = 0$, i.e., the equilibrium measure attains the lower constraint $0$.
\end{definition}
\begin{definition}\label{band}
    A \emph{band} is an open interval with $0 < \mu_0 < \sigma$, i.e., the equilibrium measure does not achieve the lower or the upper constraint.
\end{definition}
\begin{definition}\label{saturated}
    A \emph{saturated region} is an open interval with $\mu_0 = \sigma$, that is,  the equilibrium measure achieves the upper constraint.
\end{definition}

For the weight function in \eqref{weight}, there exists a critical value $c_{cr}= \frac{\pi^2}{4}$: when the parameter $c$ in \eqref{weight} is larger than $c_{cr}$,  the upper constraint in \eqref{set-M} is achieved on the left part of the support of the equilibrium measure $\mu_0$. That is, we have a ``saturated region-band-void" case when $c>c_{cr}$. When $0<c<c_{cr}$, the upper constraint does not have an influence. The equilibrium problem is similar to that of continuous cases. In the present paper, we will focus on the ``saturated region-band-void" case when  $c>c_{cr}$ and study asymptotics of the corresponding orthogonal polynomials. In the literature, such kind of problems have been studied by many researchers; see, for example, \cite{baik2007,Ble:Lie2014,dai2007,ou2010}. In our case, the origin needs some careful treatments because it is the hard edge of the equilibrium measure $\mu_0$. When $c>c_{cr}$, a local parametrix in terms of the Gamma function will be constructed near the origin in our subsequent RH analysis; see the similar constructions in \cite{lin2013,Wang:Wong2011}.  While for the case $c<c_{cr}$, since the upper constraint is not active, we expect the local parametrix near the origin is constructed in terms of the Bessel functions, which is the same as the continuous case. However, we are unable to transform the local parametrix to the standard Bessel parametrix. Moreover, when computing the zeros of the discrete Laguerre polynomials and classical Laguerre polynomials numerically, we notice that the locations of the zeros are almost the same in the interval bounded away from the origin, while there are some significant differences near the origin; see the discussion in Section \ref{sec:dl:discussion}. This is an unexpected and interesting observation. We intend to investigate this case, as well as the case when $c$ is close to $c_{cr}$, in our future study.


The rest of the paper is organized as follows. In Section \ref{sec:dl:em}, we first give the equilibrium measure $\mu_0$ for the discrete Laguerre polynomials, and the corresponding $g$-function. Next, we state our main results in Section \ref{sec:dl:mr}. In Section \ref{sec:dl:rhp}, an interpolation problem is constructed for the discrete Laguerre polynomials in the first place. Next, we remove the poles on the positive real line and  convert it into a continuous RH problem. Then, the Deift-Zhou nonlinear steepest descent method for RH problems is applied to analyze the RH problem asymptotically, where the local parametrices are constructed in terms of Gamma functions and Airy functions. In Section \ref{sec:dl:proof}, we prove our main results and obtain the asymptotics  for the orthogonal polynomials and the related constants. Finally, some discussions about the discrete and continuous Laguerre polynomials, when the upper constraint is not active, are made in Section \ref{sec:dl:discussion}.

For convenience, we choose $n=N$ in the rest of the paper.
And unless specified, we take the principle branch of $z^c$.


\section{Equilibrium measure and the $g$-function}
\label{sec:dl:em}

Let us first consider the upper constraint measure $\sigma$ in \eqref{set-M} for the discrete Laguerre polynomials. For each $N$, let $\sigma_N$ denote the counting measure of the orthogonality  lattice $L_N$ in \eqref{ortho-lattice}, normalized by the factor $1/N$, that is,
\begin{equation}
    \sigma_{ N} ([a, b]) = \frac{\#\{k : a \le \frac{k^2}{N^2} \le b\}}{N},
\end{equation}
where $0 < a < b$. The above formula is equivalent to
\begin{equation}
    \sigma_{ N} ([a, b]) = \frac{\#\{k : N\sqrt{a} \le k \le N\sqrt{b}\}}{N}.
\end{equation}
As $N \to \infty$, we have
\begin{equation}
    \lim_{N \to \infty} \sigma_{N} ([a, b]) = \sqrt{b} - \sqrt{a}= \int_a^b \frac{1}{2\sqrt{x}} dx.
\end{equation}
It then follows from the above formula that the density of the constraint measure $\sigma$ is given by
\begin{equation} \label{upper-cons}
    \sigma' (x) = \frac{1}{2\sqrt{x}}, \qquad x \in (0, +\infty).
\end{equation}
In the literature, people usually consider polynomials orthogonal on a uniform lattice. As a consequence, the upper constraint is a measure with constant density; for example, see \cite{baik2007}. In \cite{van2020}, Van Assche and  Van Baelen consider polynomials orthogonal on a $q$-lattice $\{q^k, k = 0, 1, 2, 3, \ldots \}$, with $0 < q < 1$, where the density of the upper constraint is given by $-\frac{1}{x \, \log q}$, $x \in (0,1]$. Here, $L_N$ in \eqref{ortho-lattice} gives us another example of non-uniform lattice.

Next, let us consider the equilibrium measure $\mu_0$ in \eqref{em-def} associated with the weight function \eqref{weight}. In this case, the external field in \eqref{em-def} is simply $V(x) = cx$, $c>0$. We have the following properties for the equilibrium measure $\mu_0$, where the proof will be deferred to Section \ref{sec:dl:rhp:1st}.

\begin{proposition} \label{prop:EM}
   Let $\rho(x)$ be the density of the equilibrium measure $\mu_0$. There exists a critical value
   \begin{equation}
   c_{cr} = \frac{\pi^2}{4},
   \end{equation}
   such that
\begin{equation}\label{em:rho0}
    \rho (x) = \frac{c}{2\pi} \sqrt{\frac{4-cx}{cx}}, \quad x \in \left(0, \frac{4}{c}\right], \qquad \textrm{when } 0<c<c_{cr},
\end{equation}
and
 \begin{equation}\label{em:rho2}
                \rho (x) = \begin{cases}
                    \D\frac{1}{2\sqrt{x}}, & x \in (0, a), \vspace{2pt}\\
                    \D\frac{c}{2 \pi} \sqrt{\frac{b-x}{x-a}} + \frac{1}{2 \pi} \sqrt{\frac{b-x}{x-a}}\int_0^{a} \sqrt{\frac{a-s}{s(b-s)}}\frac{ds}{s-x}, & x \in (a, b),
                \end{cases}
            \end{equation}
when $ c>c_{cr}$.  Here, the endpoints  $a \equiv a(c)$, $b \equiv b(c)$ with $0<a<b$ are uniquely determined by the following relations
        \begin{eqnarray}
        && \frac{c(b-a)}{4} + \frac{1}{2} \int_0^{a} \sqrt{\frac{a-x}{x(b-x)}} dx = 1, \label{ab1} \\
            && \sqrt{a} + \frac{c(b-a)}{4} + \frac{1}{2 \pi} \int_a^{b} \sqrt{\frac{b-x}{x-a}} \int_0^{a} \sqrt{\frac{a-s}{s(b-s)}} \frac{ds}{s-x} dx = 1.       \label{ab2}
        \end{eqnarray}
\end{proposition}

Clearly, when $c<c_{cr}$, the upper constraint $\sigma$ with density given in \eqref{upper-cons} does not have an influence on the equilibrium measure $\mu_0$. When $c > c_{cr}$, the upper constraint $\sigma$ takes effect near the origin. Then, we will have a saturated region on the left part of the support of $\mu_0$, namely $(0, a]$. Moreover, from the expression of $\rho(x)$ given in \eqref{em:rho2}, one can get the asymptotic behaviours of the density function $\rho(x)$ as $x$ tends to the endpoints $a$ and $b$. More precisely,  when $c > c_{cr}$, there exist positive constants $C_1, C_2>0$, such that
\begin{equation}\label{rho-a}
    \rho (x) = \frac{1}{2\sqrt{x}} - C_1 \sqrt{x-a} \Big(1+ O(x-a) \Big),  \qquad \textrm{as } x \to a^+,
\end{equation}
and
\begin{equation}\label{rho-b}
    \rho (x) = C_2 \sqrt{b-x} \Big(1+ O(x-b) \Big), \qquad \textrm{as } x \to b^-.
\end{equation}

With the density of the equilibrium measure given in \eqref{em:rho2}, let us define the $g$-function as follows:
\begin{equation}\label{g}
    g(z) = \int_0^{b} \log (z-s) \rho (s) ds, \qquad z \in \mathbb{C} \setminus (- \infty, b],
\end{equation}
where $\log(\cdot)$ takes the principal branch. The above $g$-function will play an important role in our subsequent RH analysis.

To state our main results in the coming section, let us introduce one constant and two more functions below: 
\begin{equation}
    l = 2 \int_0^{b} \log |a - s| \rho (s) ds - c a,
    \label{l-def}
\end{equation}
\begin{align}
    \mathcal{N}_1(z) & =  \frac{D_{\infty}}{D(z)} \frac{\gamma(z) + \gamma(z)^{-1} }{2},  \quad z \in \mathbb{C} \setminus [0, b],
     \label{n11-def}\\
    \mathcal{N}_2(z) & =  D_{\infty} D(z) \frac{\gamma(z) - \gamma(z)^{-1}}{-2 i},  \quad z \in \mathbb{C} \setminus [0, b],
    \label{n12-def}
\end{align}
where
\begin{align}
    D(z) & =  \left(\frac{(\sqrt{a}+\sqrt{b}) z}{z + \sqrt{a b} + \sqrt{(z-a)(z-b)}}\right)^{\alpha - \frac{1}{2}}, \qquad z\in  \mathbb{C} \setminus [a, b],
    \label{szego}\\
    \gamma(z) & = \frac{z^{\frac{1}{2}}}{(z-a)^{\frac{1}{4}}(z-b)^{\frac{1}{4}}} , \qquad z\in  \mathbb{C} \setminus [0, b].
    \label{gamma-def}
\end{align}
Note that $D(z) \sim  D_{\infty} $ when $z \to \infty$ with
\begin{equation} \label{Dinfty-def}
  D_{\infty} =  \left( \frac{\sqrt{a}+\sqrt{b}}{2} \right)^{\alpha - \frac{1}{2}}.
\end{equation}


\section{Main results} \label{sec:dl:mr}

Now we are ready to state our main results. First, we have the following asymptotics for the normalization constant  $h_{n,n}$ in \eqref{orthogonality} and the recurrence coefficients $\mathscr{A}_{n,n}, \mathscr{B}_{n,n}$ in \eqref{recurrence}.

\begin{theorem}\label{coefficient}
	(Asymptotics of the normalization constant and the recurrence coefficients)
When $N=n$ and $n, N \to \infty$ with $c > c_{cr} = \D\frac{\pi^2}{4}$, we have
	\begin{eqnarray}
        h_{n, n} &=&  \frac{n \pi e^{nl}(a+b)}{4}\left(\frac{\sqrt{a}+\sqrt{b}}{2}\right)^{2 \alpha-1} \left[1+O\left( \frac{1}{n} \right) \right], \label{hn-final}\\
        \mathscr{A}_{n, n}^2  &=& \frac{(a+b)^2}{16} + O\left( \frac{1}{n} \right), \label{an-final}	\\
          \mathscr{B}_{n, n} &=& \frac{a^2+b^2}{2(a +b)} + O\left( \frac{1}{n} \right), \label{bn-final}
	\end{eqnarray}
	where $a$ and $b$ are determined by \eqref{ab1}-\eqref{ab2}, and $l$ is given by \eqref{l-def}.
\end{theorem}

Next, we state the asymptotics of the polynomials $P_{n,n}(z)$ for $z$ in different regions in the complex plane.

\begin{theorem}\label{asymptoticinv}
    (Asymptotics of $P_{n, n}(z)$ in voids)
  With the functions $g(z)$ and $ \mathcal{N}_1(z)$  given in \eqref{g} and \eqref{n11-def},  we have
    \begin{equation} \label{pn-outside}
        P_{n, n}(z) = e^{n g(z)}\left[ \mathcal{N}_1(z) + O\left( \frac{1}{n} \right) \right],
    \end{equation}
    uniformly for $z$ bounded away from the support of the equilibrium measure $\mu_0$, i.e., the interval $(0, b)$.
\end{theorem}

\begin{theorem}\label{asymptoticinb}
    (Asymptotics of $P_{n, n}(x)$ in band)
    With $l$ and $D_{\infty}$ given in \eqref{l-def} and \eqref{Dinfty-def}, we have
    \begin{equation}\label{PninB-final}
        \begin{split}
            P_{n, n}(x) = & \frac{D_{\infty}e^{\frac{n}{2}(cx + l)}\sqrt{(a+b)x-a b}}{x^{\frac{\alpha}{2}}(x-a)^{\frac{1}{4}}(b-x)^{\frac{1}{4}}} \left[\cos \left((\alpha - \frac{1}{2})\arccos{\frac{x + \sqrt{a b}}{(\sqrt{a} + \sqrt{b})\sqrt{x}}}\right.\right. \\
            &\left.\left. + \arccos{\frac{x}{\sqrt{(a+b)x - a b}}} + n \pi \int_x^{b} \rho (s) ds - \frac{\pi i}{4} \right) + O\left( \frac{1}{n} \right)\right],
        \end{split}
    \end{equation}
    uniformly for $x$ in a compact subset  of the band $(a, b)$.
\end{theorem}

\begin{theorem}\label{asymptoticins}
    (Asymptotics of $P_{n, n}(x)$ in saturated region)
   There exists a $\varepsilon>0$ such that, uniformly for $x$ in a compact subset of the saturated region $(0, a)$,  we have
    \begin{equation}\label{PninS-final}
        \begin{split}
        P_{n, n}(x) = & (-1)^{n+1} e^{n \int_0^{b} \log |x-s| \rho (s) ds}\left[ \sin{(n \pi \sqrt{x})}\frac{x-\sqrt{(a-x)(b-x)}}{\sqrt{x}(a-x)^{\frac{1}{4}}(b-x)^{\frac{1}{4}}}\right.\\
        & \left.\times \left(\frac{x + \sqrt{a b} -\sqrt{(a-x)(b-x)}}{2x}\right)^{\alpha - \frac{1}{2}}\left(1+O\left( \frac{1}{n} \right)\right) + O(e^{-n \varepsilon})\right].
        \end{split}
    \end{equation}
\end{theorem}

\begin{remark}
    Asymptotic results in the above theorems also hold in a neighborhood in the complex plane. More precisely, there exists a $\delta>0$ such that \eqref{PninB-final} and \eqref{PninS-final} are valid in $\{z\hspace{0.25em}|\hspace{0.25em} a + \delta \le \re{z} \le b-\delta, -\delta \le \im{z} \le \delta\}$ and $\{z\hspace{0.25em}|\hspace{0.25em} \delta \le \re{z} \le a-\delta, -\delta \le \im{z} \le \delta\}$, respectively.
\end{remark}

Note that the above asymptotic results do not hold in the neighbourhood of the endpoints 0, $a$ and $b$. Indeed, some special functions will appear in the asymptotics of $P_{n,n}(z)$, such as the Gamma function near the origin, and the Airy functions near the endpoints $a$ and $b$.

To state the asymptotic result near the origin, let us introduce the following function
\begin{equation}\label{H*-def}
    H^*(z)=H(z) \times \begin{cases}
        (1 - e^{2 i n \pi \sqrt{z}}), & z \in \mathbb{C}_+,\\
        (1 - e^{-2 i n \pi \sqrt{z}}), & z \in \mathbb{C}_-,
    \end{cases} \quad z \in \mathbb{C} \setminus (- i \infty, 0] \cup [0, \infty),
\end{equation}
where
\begin{equation}\label{H-def}
	H(z)= \frac{e^{n z^{\frac{1}{2}}} \Gamma(n z^{\frac{1}{2}})}{\sqrt{2 \pi} (n z^{\frac{1}{2}})^{n z^{\frac{1}{2}}- \frac{1}{2}}}, \quad z \in \mathbb{C} \setminus (- i \infty, 0].
\end{equation}
In the above formulas, we take $\arg z \in (-\frac{\pi}{2}, \frac{3 \pi}{2})$ for $z^{\frac{1}{2}}$ and $\arg z \in (- \pi, \pi)$ for $\sqrt{z}$.

\begin{theorem}\label{asymptotic0}
    (Asymptotics of $P_{n, n}(z)$ near $0$)
    There exists a $\delta>0$, such that for $|z| < \delta$, we have
    \begin{equation} \label{Pnn-void-main}
        P_{n, n}(z) = e^{n g(z)} \mathcal{N}_1(z) H^{*}(z)  \left(I + O\left( \frac{1}{n} \right) \right),
    \end{equation}
   where $\mathcal{N}_1(z)$ is given in \eqref{n11-def}, and the function $H^*(z)$ is given in \eqref{H*-def}.
\end{theorem}

\begin{remark}
    Although the functions $g(z)$, $\mathcal{N}_1(z)$ and $H^{*}(z)$ are not analytic on the interval $(0, \delta)$, one can show that the approximation \eqref{Pnn-void-main} holds  for $x \in (0, \delta)$. Let $f_{\pm}(x)$ denote the limiting values of $f(z)$ as $z \to x$  from the upper and lower half plane. Then, for $x \in (0, \delta)$, we have
    \begin{equation}
        P_{n, n}(x) = 2 i (-1)^{n+1}\sin{(n \pi \sqrt{x})} e^{n \int_0^{b} \log |x - s| \rho (s) ds} \mathcal{N}_{1, +}(x) H(x)  \left(I + O\left( \frac{1}{n} \right) \right),
    \end{equation}
    and
    \begin{equation}
        P_{n, n}(x) = -2 i (-1)^{n+1}\sin{(n \pi \sqrt{x})} e^{n \int_0^{b} \log |x - s| \rho (s) ds} \mathcal{N}_{1, -}(x) H(x)  \left(I + O\left( \frac{1}{n} \right) \right),
    \end{equation}
     where  $H(x)$ is given in \eqref{H-def}. From the definition of $\mathcal{N}_1(z)$ in \eqref{n11-def}, one can verify that $\mathcal{N}_{1, +}(x) = - \mathcal{N}_{1, -}(x)$ for $x \in (0, \delta)$. Therefore, the above two formulas indeed agree with each other.
\end{remark}

Near the band-void edge point $b$, the asymptotic expansion is given in terms of the Airy functions $\Ai(\cdot)$, and near the saturated region-band edge point $a$,  both the $\Ai(\cdot)$ and $\Bi(\cdot)$ functions appear in the asymptotic expansion. Let us introduce the following functions $f(z)$ and $\widetilde{f}(z)$, which are analytic in the neighbourhood of $b$ and $a$, respectively,
\begin{eqnarray}
f(z) &=& - \left(\frac{3\pi}{2} \int_z^{b} \rho(s)ds\right)^{\frac{2}{3}}, \qquad  z \in \mathbb{C} \setminus (-\infty, a], \label{f-def} \\
\widetilde{f}(z)& =& \left(\frac{3\pi}{2} \int_{a}^z \left(\frac{1}{2\sqrt{s}}-\rho (s)\right)ds\right)^{\frac{2}{3}}, \qquad  z \in \mathbb{C} \setminus (-\infty, 0] \cup [b, \infty).  \label{f2-def}
\end{eqnarray}
The analyticity of  $f(z)$ and $\widetilde{f}(z)$ in the neighbourhood of $b$ and $a$ follows from the asymptotics of $\rho(x)$ near the endpoints in \eqref{rho-a} and \eqref{rho-b}, respectively.

\begin{theorem}\label{asymptoticb}
    (Asymptotics of $P_{n, n}(z)$ at band-void edge point $b$)
   There exists a $\delta>0$, such that for $|z-b|<\delta$, we have
    \begin{equation}\label{Pnn-b-main}
        \begin{split}
            P_{n, n}(z)
            = & \sqrt{\pi} e^{\frac{n}{2}(cz+l)} \left [n^{\frac{1}{6}} f(z)^{\frac{1}{4}} \Ai(n^{\frac{2}{3}}f(z))\left(\mathcal{N}_1(z)-iz^{-\alpha +\frac{1}{2}}\mathcal{N}_2(z) + O\left( \frac{1}{n} \right)\right) \right. \\
             & \left. -n^{-\frac{1}{6}}f(z)^{-\frac{1}{4}}\Ai'(n^{\frac{2}{3}}f(z))\left(\mathcal{N}_1(z)+iz^{-\alpha +\frac{1}{2}}\mathcal{N}_2(z) + O\left( \frac{1}{n} \right)\right) \right],
        \end{split}
    \end{equation}
    where the constant $l$ is given in \eqref{l-def} and the functions $\mathcal{N}_1(z)$, $\mathcal{N}_2(z)$ and $f(z)$ are given in \eqref{n11-def}, \eqref{n12-def} and \eqref{f-def}, respectively.
\end{theorem}

\begin{theorem}\label{asymptotica}
    (Asymptotics of $P_{n, n}(z)$ at saturated region-band edge point $a$)
    There exists a $\delta>0$, such that for $|z-a|<\delta$, we have
    \begin{equation}\label{Pnn-a-main}
        \begin{split}
            P_{n, n}(z)
            = & (-1)^n \sqrt{\pi} e^{\frac{n}{2}(cz+l)} \left [n^{\frac{1}{6}} (-\widetilde{f}(z))^{\frac{1}{4}} \eta_1(z)\left (i \mathcal{N}_1(z)+z^{-\alpha +\frac{1}{2}} \mathcal{N}_2(z) + O\left( \frac{1}{n} \right)\right) \right. \\
             &\left. +n^{-\frac{1}{6}} (-\widetilde{f}(z))^{-\frac{1}{4}}\eta_2(z)\left(i \mathcal{N}_1(z)-z^{-\alpha +\frac{1}{2}}\mathcal{N}_2(z) + O\left( \frac{1}{n} \right)\right)\right],
        \end{split}
    \end{equation}
     where the constant $l$ is given in \eqref{l-def} and the functions $\mathcal{N}_1(z)$, $\mathcal{N}_2(z)$ and $\widetilde{f}(z)$ are given in \eqref{n11-def}, \eqref{n12-def} and \eqref{f2-def}, respectively, and
    \begin{eqnarray}
    \eta_1(z) &=& \cos{(n \pi \sqrt{z})} \Ai (-n^{\frac{2}{3}}\widetilde{f}(z))- \sin{(n \pi \sqrt{z})} \Bi(-n^{\frac{2}{3}}\widetilde{f}(z)), \\
     \eta_2(z) &=& \cos{(n \pi \sqrt{z})} \Ai' (-n^{\frac{2}{3}}\widetilde{f}(z))- \sin{(n \pi \sqrt{z})} \Bi'(-n^{\frac{2}{3}}\widetilde{f}(z)).
    \end{eqnarray}
\end{theorem}


The main tool to prove the above theorems is the RH analysis, which will be conducted in the next section.

\section{Riemann-Hilbert analysis}
\label{sec:dl:rhp}

\subsection{Interpolation problem}
\label{sec:dl:ip}
For the discrete orthogonal polynomials \eqref{Pn,n}, we construct the interpolation problem for a $2 \times 2$ matrix-valued function $\mathbf{P}(z)$ with following properties:
\begin{itemize}
    \item \emph{Analyticity}: $\mathbf{P}(z)$ is analytic for $z \in \mathbb{C} \setminus L_{N}$.
    \item \emph{Residues at poles}: At each node $x \in L_{N}$, entries in the first column of the matrix $\mathbf{P}(z)$ are analytic functions of $z$, while those in the second column have a simple pole at $x$ with residues given by
    \begin{equation}
        \Res_{z=x} \Big[ \mathbf{P}(z) \Big]_{j2}=w(x) \Big[\mathbf{P}(x)\Big]_{j1}, \quad  j=1,2.
    \end{equation}
    \item \emph{Asymptotics at infinity}: $\mathbf{P}(z)$ admits the asymptotic expansion,
    \begin{equation}\label{Pinfty}
        \mathbf{P}(z) = \left(I+\frac{\mathbf{P}_1}{z}+\frac{\mathbf{P}_2}{z^2}+O\left(\frac{1}{z^3}\right)\right)
        \begin{pmatrix}
            z^n & 0 \\
            0 & z^{-n}
        \end{pmatrix}, \qquad \textrm{as } z \to \infty.
    \end{equation}
\end{itemize}
By the well-known theorem of Fokas, Its and Kitaev \cite{fokas-its-kitaev}, the solution to the above problem is given as follows.
\begin{proposition}\label{ipsolution}
    The interpolation problem has  a unique solution:
    \begin{equation}\label{P-sol}
        \mathbf{P}(z)=
        \begin{pmatrix}
        P_{n, n}(z) & \frac{1}{n}C(wP_{n, n})(z) \\
        \frac{1}{h_{n-1, n}}P_{n-1, n}(z) & \frac{1}{nh_{n-1, n}}C(wP_{n-1, n})(z)
        \end{pmatrix}
    \end{equation}
    with the discrete Cauchy transformation
    \begin{equation}
        C(f)(z)=\sum_{x \in L_{N}} \frac{f(x)}{z-x}.
    \end{equation}
\end{proposition}
Recalling the orthogonality relation \eqref{orthogonality}, we have, as $z \to \infty$,
\begin{equation}
    C(wP_{n, n})(z) = \sum_{x \in L_{N}} \frac{w(x)P_{n, n}(x)}{z-x} \sim \sum_{x \in L_{N}} w(x) P_{n, n}(x) \sum_{j=0}^{\infty} \frac{x^j}{z^{j+1}} = \frac{h_{n, n}}{z^{n+1}} + O\left( \frac{1}{z^{n+2}} \right),
\end{equation}
which justifies asymptotic expansion \eqref{Pinfty} and implies that
\begin{equation}\label{hN}
    h_{n, n} = [\mathbf{P}_1]_{12}, \quad h_{n-1, n}^{-1} = [\mathbf{P}_1]_{21}.
\end{equation}
Moreover, the recurrence coefficients in \eqref{recurrence} are given by
\begin{equation}\label{gammabeta}
    \mathscr{A}_{n, n}^2 = [\mathbf{P}_1]_{12}[\mathbf{P}_1]_{21}, \quad \mathscr{B}_{n, n} = \frac{[\mathbf{P}_2]_{12}}{[\mathbf{P}_1]_{12}} + [\mathbf{P}_1]_{11}.
\end{equation}
Next, we follow the approach in Bleher and Liechty \cite{bleher2011} to convert the interpolation problem to a continuous RH problem, where the condition on poles and residues is replaced by jump conditions on certain contours in the complex plane.

\subsection{Riemann-Hilbert problem}
In the first step, we remove the poles on the real line and transform the discrete interpolation problem to a continuous RH problem. Let $\Pi (z)$ be defined as
\begin{equation}
    \Pi (z) = \frac{2 \sqrt{z} \sin {(n\pi \sqrt{z})}}{n \pi},
\end{equation}
which is an entire function. Moreover, for  $x_k \in L_N$, we have
\begin{equation}
    \Pi (x_k) = 0, \quad \Pi '(x_k) = \cos{(n\pi \sqrt{x_k}}) = (-1)^k.
\end{equation}
Let us further introduce the following matrix-valued functions:
\begin{equation} \label{yu-def}
    \mathbf{Y}^u(z) = \mathbf{P}(z) \times \begin{cases}
                        D_{+}^u(z), & \textrm{$\im z \ge 0$}, \\
                        D_{-}^u(z), & \textrm{$\im z \le 0$},
                        \end{cases}
\end{equation}
and
\begin{equation}
    \mathbf{Y}^l(z) = \mathbf{P}(z) \times \begin{cases}
                        D_{+}^l(z), & \textrm{$\im z \ge 0$}, \\
                        D_{-}^l(z), & \textrm{$\im z \le 0$},
                        \end{cases}
\end{equation}
where $D_{\pm}^u(z)$ are  upper and lower triangular matrices given below:
\begin{equation}\label{Du}
    D_{\pm}^u(z)=
    \begin{pmatrix}
        1 & -\frac{w(z)}{\Pi(z)}e^{\pm i n \pi \sqrt{z}}\\
        0 & 1
    \end{pmatrix}
\end{equation}
and
\begin{equation}\label{Dl}
    D_{\pm}^l(z)=
    \begin{pmatrix}
        \frac{\sqrt{z}}{\Pi (z)} & 0\\
        -\frac{\sqrt{z}}{w(z)}e^{\pm i n \pi \sqrt{z}} & \frac{\Pi (z)}{\sqrt{z}}
    \end{pmatrix}.
\end{equation}
The functions $\mathbf{Y}^u(z)$ and $\mathbf{Y}^l(z)$ are meromorphic on the closed upper and lower complex planes and are two-valued on the real axis. Although $\mathbf{P}(z)$ has poles on the lattice $L_N$, one can show that after multiplying the functions $D_{\pm}^u(z)$ and $D_{\pm}^l(z)$, the poles in $\mathbf{Y}^u(z)$ and $\mathbf{Y}^l(z)$ are cancelled.

Let us introduce the following transformation
\begin{equation}\label{iptorhp}
    \mathbf{Y}(z) = \begin{cases}
            K\mathbf{Y}^u(z)K^{-1}, & \textrm{$z \in \Omega _{\pm}^{\bigtriangledown}$},\\
            K\mathbf{Y}^l(z)K^{-1}, & \textrm{$z \in \Omega _{\pm}^{\bigtriangleup}$},\\
            K\mathbf{P}(z)K^{-1}, & \textrm{otherwise},
            \end{cases}
\end{equation}
where the regions $\Omega _{\pm}^{\bigtriangledown}$ and $\Omega _{\pm}^{\bigtriangleup}$ are depicted in Figure \ref{sigmaY}, and $K$ is the following constant matrix
\begin{equation}\label{K}
    K=\begin{pmatrix}
            1 & 0\\
            0 & -i n \pi
        \end{pmatrix}.
\end{equation}
Then, $\mathbf{Y}(z)$ satisfies the following proposition.
\begin{figure}[h]
    \centering
    \includegraphics[width=0.6\textwidth]{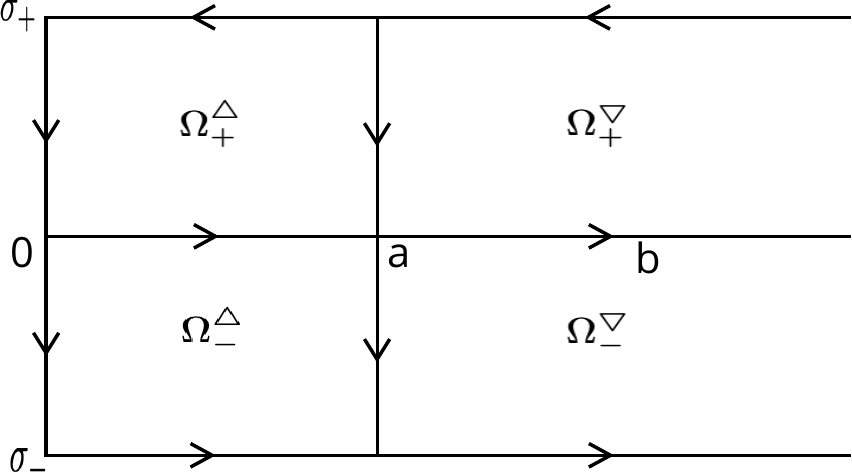}
    \caption{Contour $\Sigma_{\mathbf{Y}}$ and regions $\Omega _{\pm}^{\bigtriangledown}, \Omega _{\pm}^{\bigtriangleup}$, where $\sigma_{\pm} = (0, \pm i \varepsilon) \cup \{(0, a) \pm i \varepsilon \}$.}
    \label{sigmaY}
\end{figure}

\begin{proposition}
    The function $\mathbf{Y}(z)$ defined in \eqref{iptorhp} satisfies a RH problem as follows:
    \begin{itemize}
    \item {\emph{Analyticity}: $\mathbf{Y}(z)$ is analytic for $z \in \mathbb{C} \setminus \Sigma_{\mathbf{Y}}$.}
    \item{\emph{Jump condition}: $\mathbf{Y}_{+}(z)=\mathbf{Y}_{-}(z)J_{\mathbf{Y}}(z)$ for $z \in \Sigma_\textbf{Y}$ with
    \begin{equation} \label{Y-jump}
        J_{\mathbf{Y}}(z)=
        \begin{cases}
            \begin{pmatrix}
                1 & x^{\alpha-\frac{1}{2}}e^{-n c x}\\
                0 & 1
            \end{pmatrix}, & x \in (a, \infty), \\
            \begin{pmatrix}
                1 & 0\\
                -n^2 \pi^2 x^{-\alpha + \frac{1}{2}}e^{n c x} & 1
            \end{pmatrix}, & x \in (0,a), \\
            \begin{pmatrix}
                \frac{\sqrt{z}}{\Pi (z)} & 0\\
                in\pi \frac{\sqrt{z}}{w(z)}e^{\pm i n \pi \sqrt{z}} & \frac{\Pi (z)}{\sqrt{z}}
            \end{pmatrix}, & z \in \sigma_{\pm}, \\
            \begin{pmatrix}
                1 & \frac{w(z)}{in \pi \Pi(z)}e^{\pm i n \pi \sqrt{z}}\\
                0 & 1
            \end{pmatrix}, & z \in (a, \infty) \pm i \varepsilon, \\
            \begin{pmatrix}
                \frac{\Pi (z)}{\sqrt{z}} & \frac{w(z)}{in \pi \sqrt{z}}e^{\pm i n \pi \sqrt{z}}\\
                -in\pi \frac{\sqrt{z}}{w(z)}e^{\pm i n \pi \sqrt{z}} & \mp in \pi e^{\pm i n \pi \sqrt{z}}
            \end{pmatrix}, & z \in (a, a \pm i \varepsilon).
        \end{cases}
    \end{equation}
    }
    \item{\emph{Asymptotics at infinity}:
    \begin{equation}
        \mathbf{Y}(z) = \left(I + \frac{\mathbf{Y}_1}{z}+\frac{\mathbf{Y}_2}{z^2}+O\left(\frac{1}{z^3}\right)\right)
        \begin{pmatrix}
            z^n & 0 \\
            0 & z^{-n}
        \end{pmatrix}, \qquad \textrm{as } z \to \infty.
    \end{equation}}
    \item{\emph{Asymptotics at the origin}: as $z \to 0$,
    \begin{equation} \label{Y-origin}
        \mathbf{Y}(z) =
        \begin{cases}
            \begin{pmatrix}
                O(1) & O(1) \\
                O(1) & O(1)
            \end{pmatrix}, & \textrm{for $\re{z} \le 0$,}\\
            \begin{pmatrix}
                O(|z|^{-\frac{1}{2}}) & O(|z|^{\frac{1}{2}}) \\
                O(|z|^{-\frac{1}{2}}) & O(|z|^{\frac{1}{2}})
            \end{pmatrix}, & \textrm{for $\re{z}> 0$ and $\alpha \le 1$,}\\
            \begin{pmatrix}
                O(|z|^{\frac{1}{2} - \alpha}) & O(|z|^{\frac{1}{2}}) \\
                O(|z|^{\frac{1}{2} - \alpha}) & O(|z|^{\frac{1}{2}})
            \end{pmatrix}, & \textrm{for $\re{z}> 0$ and $\alpha > 1$.}
        \end{cases}
    \end{equation}}
\end{itemize}
\end{proposition}
\begin{proof}

It is clear that $\mathbf{Y}(z)$ is analytic for $z \in \mathbb{C} \setminus \Sigma_{\mathbf{Y}}$ from its definition in \eqref{iptorhp}. The jumps on $\Sigma_{\mathbf{Y}} \setminus \mathbb{R}$ are easy to verify.
    For $x \in (0,a)$, we have
    \begin{equation}
        \begin{split}
            J_{\mathbf{Y}}(x) & = \mathbf{Y}_{-}(x)^{-1} \mathbf{Y}_{+}(x)\\
            & = K D_{-}^l(x)^{-1} D_{+}^l(x) K^{-1}.
        \end{split}
    \end{equation}
    Substituting \eqref{Dl}  and \eqref{K}  into the above equation, we obtain
    \begin{equation}
        \begin{split}
            J_{\mathbf{Y}}(x) & = \begin{pmatrix}
            1 & 0\\
            0 & -i n \pi
        \end{pmatrix} \begin{pmatrix}
        \frac{\Pi (x)}{\sqrt{x}} & 0\\
        \frac{\sqrt{x}}{w(x)}e^{- i n \pi \sqrt{x}} &  \frac{\sqrt{x}}{\Pi (x)}
    \end{pmatrix} \begin{pmatrix}
        \frac{\sqrt{x}}{\Pi (x)} & 0\\
        -\frac{\sqrt{x}}{w(x)}e^{i n \pi \sqrt{x}} & \frac{\Pi (x)}{\sqrt{x}}
    \end{pmatrix} \begin{pmatrix}
            1 & 0\\
            0 & -\frac{1}{i n \pi}
        \end{pmatrix}\\
        & = \begin{pmatrix}
                1 & 0\\
                -n^2 \pi^2 x^{-\alpha + \frac{1}{2}}e^{ncx} & 1
            \end{pmatrix}.
        \end{split}
    \end{equation}
    Similarly, for $x \in (a, \infty)$, we have
    \begin{equation}
        J_{\mathbf{Y}}(x) = K D_{-}^u(x)^{-1} D_{+}^u(x) K^{-1}.
    \end{equation}
    Substituting \eqref{Du}  and \eqref{K}  into the above equation, we get
    \begin{equation}
        \begin{split}
            J_{\mathbf{Y}}(x) & = \begin{pmatrix}
            1 & 0\\
            0 & -i n \pi
        \end{pmatrix}\begin{pmatrix}
        1 & \frac{w(x)}{\Pi(x)}e^{- i n \pi \sqrt{x}}\\
        0 & 1
    \end{pmatrix} \begin{pmatrix}
        1 & -\frac{w(x)}{\Pi(x)}e^{ i n \pi \sqrt{x}}\\
        0 & 1
    \end{pmatrix} \begin{pmatrix}
            1 & 0\\
            0 & -\frac{1}{i n \pi}
        \end{pmatrix}\\
        & = \begin{pmatrix}
                1 & x^{\alpha-\frac{1}{2}}e^{-ncx}\\
                0 & 1
            \end{pmatrix}.
        \end{split}
    \end{equation}
    Note that, since $D_{\pm}^u(z)$ tend to the identity matrix as $z \to \infty$, $\mathbf{Y}(z)$ has the same asymptotic behaviour as $\mathbf{P}(z)$ when $z \to \infty$.

    Next, we consider the behaviour of $\mathbf{Y}(z)$ near the origin. From \eqref{P-sol}, we have
    \begin{equation}
        \mathbf{P}(z) = \begin{pmatrix}
            O(1) & O(1) \\
            O(1) & O(1)
        \end{pmatrix} \quad \textrm{as $z \to 0$}.
    \end{equation}
    By the definitions of  $D_{\pm}^l(z)$  in \eqref{Dl}, we get
    \begin{equation}
        D_{\pm}^l(z) = \begin{pmatrix}
                O(|z|^{-\frac{1}{2}}) & 0 \\
                O(|z|^{\frac{1}{2} - \alpha}) & O(|z|^{\frac{1}{2}})
            \end{pmatrix} \quad \textrm{as $z \to 0$}.
    \end{equation}
    Then, a straightforward calculation gives us the asymptotic behaviour of $\mathbf{Y}(z)$ near the origin in \eqref{Y-origin}.

    This finishes the proof of our proposition.
\end{proof}

\begin{remark}\label{changeweight}
The weight function for orthogonal polynomials is embedded into the off-diagonal entries of $J_{\mathbf{Y}}(z)$ for $x \in \mathbb{R}$ in \eqref{Y-jump}. When the discrete interpolation problem is transformed to a continuous RH problem, an interesting observation is that the original weight function $x^{\alpha}e^{-ncx}$ in \eqref{weight} is modified to $x^{\alpha-\frac{1}{2}}e^{-ncx}$, where the exponent of $x$ decreases from $\alpha$ to $\alpha - \frac{1}{2}$.
\end{remark}

In the subsequent sections, we apply the Deift-Zhou nonlinear steepest descent method for RH problems to study the asymptotics of $\mathbf{Y}(z)$ as $n \to \infty$.

\subsection{First transformation}
\label{sec:dl:rhp:1st}

In the first transformation, we normalize $\mathbf{Y}(z)$ at infinity with the aid of the $g$-function defined in \eqref{g}. Before we introduce the first transformation, let us first study some properties of $g(z)$ and prove Proposition \ref{prop:EM} about the equilibrium measure $\mu_0$.

\begin{proposition}
\label{prop-g-func}
The $g$-function defined in \eqref{g} satisfies the following properties:
\begin{itemize}
\item[(1)] From the variational conditions \eqref{variational}, we have
\begin{equation}\label{2g-V-l}
        g_{+}(x) + g_{-}(x) \begin{cases}
        > cx + l, & \textrm{for $x \in (0, a)$},\\
        = cx + l, & \textrm{for $x \in [a, b]$},\\
        < cx + l, & \textrm{for $x \in (b, \infty)$}.
        \end{cases}
    \end{equation}
    \item[(2)] For $x \in \mathbb{R}$, we have
    \begin{equation}\label{gproperty}
        g_{+}(x) - g_{-}(x) = \begin{cases}
        2 \pi i, & \textrm{for $x \in (-\infty, 0),$}\\
        2 \pi i (1 - \sqrt{x}), & \textrm{for $x \in (0, a)$,}\\
        2 \pi i \int_x^{b} \rho (s) ds, & \textrm{for $x \in (a, b)$,}\\
        0, & \textrm{for $x \in (b, \infty)$.}
        \end{cases}
    \end{equation}
    \item[(3)] As $z \to \infty$, $g(z) =\D \log z + O\left(\frac{1}{z}\right).$
\end{itemize}
\end{proposition}

\begin{proof}

Recalling the definition of $g(z)$ in \eqref{g}, where the principal branch is chosen for $\log(x-s)$ for $x \in \mathbb{R}$. Then,  we have
    \begin{equation} \label{g+_-1}
        g_{\pm}(x)  = \int_0^{b} (\log|x-s| \pm \pi i) \rho (s) ds
        =  \int_0^{b} \log|x-s| \rho (s) ds \pm \pi i, \quad x \in (-\infty, 0).
    \end{equation}
Similarly, we get
    \begin{equation} \label{g+_-2}
            g_{\pm}(x) = \int_0^{b} \log|x-s| \rho (s) ds \pm \pi i \int _x^{b} \rho (s) ds, \qquad x \in (0, b).
    \end{equation}
It then follows from the above formulas that
\begin{equation}
g_+(x) + g_-(x) = 2 \int_0^{b} \log|x-s| \rho (s) ds, \qquad x \in (0,\infty).
\end{equation}
Recalling the following variational conditions when $c > c_{cr}$:
\begin{equation}\label{variational}
    2 \int \log |x-y| \rho (y) dy - cx - l \begin{cases}
        > 0, & \textrm{for $x \in (0, a)$},\\
        = 0, & \textrm{for $x \in [a, b]$},\\
        < 0, & \textrm{for $x \in (b, \infty)$},
    \end{cases}
\end{equation}
we obtain \eqref{2g-V-l} from the above two formulas.
Note that we are in the ``saturated region-band-void" situation, where the upper constraint \eqref{upper-cons} is achieved on $(0,a)$. Then, when  $0<x<a$, the last integral in formula \eqref{g+_-2} can be rewritten as
\begin{equation}
        \int _x^{b} \rho (s) ds = \int_0^{b} \rho (s) ds - \int_0^{x} \rho (s) ds = 1- \sqrt{x}, \qquad x \in (0 ,a).
    \end{equation}
This gives us
    \begin{equation} \label{g+_-3}
        g_{\pm}(x)  = \int_0^{b} \log|x-s| \rho (s) ds \pm \pi i (1- \sqrt{x}), \qquad x \in (0 ,a).
    \end{equation}
Then, the relation \eqref{gproperty} is a straightforward consequence of \eqref{g+_-1}, \eqref{g+_-2}, \eqref{g+_-3}, and the fact that $g(x)$ is analytic for $x>b$. Since $\int_0^b\rho(x)=1$, the large-$z$ behaviour of $g(z)$ follows immediately.

This finishes the proof of our proposition.
\end{proof}

We are now in a position to complete the proof of Proposition \ref{prop:EM} in Section \ref{sec:dl:em}.

\noindent\emph{Proof of Proposition \ref{prop:EM}.}
When $c < c_{cr}$, the equilibrium measure is the same as the continuous case. The density function \eqref{em:rho0} is well-known; for example, see \cite{qiu2008}. Here, we focus on the case when $c>c_{cr}$.

Recall the definition of $g(z)$ in \eqref{g}, where the derivative is given by
\begin{equation}
    g'(z) = \int_0^{b} \frac{\rho (s)}{z-s} dx, \qquad z \in \mathbb{C} \setminus [0, b].
\end{equation}
To get information of $g'_{\pm}(x)$ when $x \in (0, b)$, we first apply the Plemelj-Sokhotsky formula and obtain
\begin{equation}\label{em:rho}
    g'_{+} (x) - g'_{-} (x) = -2 \pi i \rho (x), \qquad \textrm{for $x \in (0, b)$}.
\end{equation}
Moreover, differentiating \eqref{2g-V-l} with respect to $x$, we get
\begin{equation}
    g'_{+} (x) + g'_{-} (x) = c, \qquad \textrm{for $x \in (a, b)$}.
\end{equation}
Then, the expression of $\rho(x)$ can be determined from the above two formulas by considering an associated scaler RH problem.

For this purpose, let us introduce
\begin{equation}\label{G-def}
    G(z) = \sqrt{\frac{z-b}{z-a}}, \qquad z  \in \mathbb{C} \setminus [a, b],
\end{equation}
where the branch is chosen such that $G(z) \sim 1$ as $z \to \infty$. For $x \in (0, b)$, it follows from the above definition that
\begin{equation}
G(x) = \sqrt{\frac{b-x}{a-x}}, \quad x \in (0, a); \qquad G_{\pm}(x) = \pm i \sqrt{\frac{b-x}{x-a}}, \qquad x \in (a, b).
\end{equation}
Then, it is easily verified that the function
\begin{equation}\label{em:omega}
    \omega (z) = \frac{g'(z)}{G(z)}, \qquad z \in \mathbb{C} \setminus [0, b],
\end{equation}
satisfies a scalar RH problem as follows.
\begin{itemize}
    \item {\emph{Analyticity}: $\omega (z)$ is analytic for $z \in \mathbb{C} \setminus [0, b]$.}
    \item {\emph{Jump condition}:
    \begin{equation}
        \omega_+(x) - \omega_-(x) =
        \begin{cases}
            -\pi i \sqrt{\frac{a-x}{x(b-x)}}, & x \in (0, a),\\
            -c i \sqrt{\frac{x-a}{b-x}}, & x \in (a, b).
        \end{cases}
    \end{equation}}
    \item {\emph{Asymptotics at infinity}:
    \begin{equation}
        \omega (z) = \frac{1}{z} + O\left(\frac{1}{z^2}\right), \qquad \textrm{as } z \to \infty.
    \end{equation}}
\end{itemize}
With the Plemelj-Sokhotsky formula, the solution to the above RH problem is given by
\begin{eqnarray}
\omega (z) &=& \frac{1}{2 \pi i} \int_0^{a} \left( -\pi i\sqrt{\frac{a-s}{s(b-s)}} \right) \frac{ds}{s-z} + \frac{1}{2 \pi i} \int_{a}^{b} \left(-c i \sqrt{\frac{s-a}{b-s}}\right) \frac{ds}{s-z} \nonumber \\
& = & \frac{c}{2} \sqrt{\frac{z-a}{z-b}} - \frac{c}{2} -\frac{1}{2} \int_0^{a} \sqrt{\frac{a-s}{s(b-s)}} \frac{ds}{s-z}. \label{em:omega2}
\end{eqnarray}
To get the expression of $\rho(x)$, we obtain from  \eqref{em:rho} and \eqref{em:omega} that
\begin{equation}
    \rho (x) = -\frac{1}{2 \pi i} \Big(\omega_+ (x) G_+(x) - \omega_- (x) G_-(x)  \Big),
\end{equation}
Substituting \eqref{G-def} and \eqref{em:omega2} into the above equation, we get the density function \eqref{em:rho2}.

To determine the endpoints $a$ and $b$, we consider the asymptotics of $\omega (z)$ as $z \to \infty$. From \eqref{em:omega2}, we have
\begin{equation}
    \omega (z) = \frac{1}{z} \left(\frac{c(b-a)}{4} + \frac{1}{2} \int_0^{a} \sqrt{\frac{a-s}{s(b-s)}} ds \right) + O\left(\frac{1}{z^2}\right), \qquad z \to \infty.
\end{equation}
Then, the relation \eqref{ab1} follows from the above formula and the  requirement $\omega (z)  \sim \frac{1}{z}$. The second relation \eqref{ab2} is a direct consequence of the normalization condition $\int_0^{b} \rho(s) ds = 1$.

This finishes the proof of Proposition \ref{prop:EM}.
\hfill \qed

\medskip

Now, let us introduce the first transformation as follows:
\begin{equation}\label{ytot}
    \mathbf{T}(z)=e^{\frac{nl}{2}\sigma_3} \mathbf{Y}(z) e^{-n(g(z)-\frac{l}{2})\sigma_3},
\end{equation}
where $l$ is the Lagrange multiplier given in \eqref{l-def} and $\sigma_3=\begin{pmatrix}
1 & 0\\
0 & -1
\end{pmatrix}$ is the third Pauli matrix. It is easy to verify that $\mathbf{T}(z)$ satisfies the following RH problem:
\begin{itemize}
    \item {\emph{Analyticity}: $\mathbf{T}(z)$ is analytic for $z \in \mathbb{C} \setminus \Sigma_{\mathbf{Y}}$.}
    \item{\emph{Jump condition}: $\mathbf{T}_{+}(z)=\mathbf{T}_{-}(z)J_{\mathbf{T}}(z)$ for $z \in \Sigma_{\mathbf{Y}}$,
    \begin{equation}
        J_{\mathbf{T}}(z)=
        \begin{cases}
            \begin{pmatrix}
                e^{n(g_{-}(x) - g_{+}(x))} & x^{\alpha-\frac{1}{2}}\\
                0 & e^{n(g_{+}(x) - g_{-}(x))}
            \end{pmatrix}, & x \in (a, b), \\
            \begin{pmatrix}
                1 & x^{\alpha-\frac{1}{2}}e^{n(g_{+}(x) + g_{-}(x) - cx -l)}\\
                0 & 1
            \end{pmatrix}, & x \in (b, \infty), \\
            \begin{pmatrix}
                e^{n(g_{-}(x) - g_{+}(x))} & 0\\
                -n^2 \pi^2 x^{-\alpha + \frac{1}{2}}e^{-n(g_{+}(x) + g_{-}(x) - cx -l)} & e^{n(g_{+}(x) - g_{-}(x))}
            \end{pmatrix}, & x \in (0,a), \\
            \begin{pmatrix}
                \frac{n\pi}{2 \sin{(n \pi \sqrt{z})}} & 0\\
                i n \pi z^{-\alpha + \frac{1}{2}} e^{-n(2g(z) - cz -l) \pm i n \pi \sqrt{z}} & \frac{2 \sin{(n \pi \sqrt{z})}}{n\pi}
            \end{pmatrix}, & z \in \sigma_{\pm}, \\
            \begin{pmatrix}
                1 & \pm z^{\alpha-\frac{1}{2}} \frac{e^{n(2g(z) - cz -l)}}{1-e^{\mp 2in \pi \sqrt{z}}}\\
                0 & 1
            \end{pmatrix}, & z \in (a, \infty) \pm i \varepsilon,\\
            \begin{pmatrix}
                \frac{2 \sin{(n \pi \sqrt{z})}}{n\pi} & \frac{z^{\alpha - \frac{1}{2}}}{in\pi}  e^{n(2g(z) - cz -l) \pm i n \pi \sqrt{z}}\\
                -i n \pi z^{-\alpha + \frac{1}{2}} e^{-n(2g(z) - cz -l) \pm i n \pi \sqrt{z}} & \mp i n \pi e^{\pm i n \pi \sqrt{z}}
            \end{pmatrix}, & z \in (a, a \pm i \varepsilon).
        \end{cases}
    \end{equation}
    }
    \item{\emph{Asymptotics at infinity}:
    \begin{equation}
        \mathbf{T}(z) = I + \frac{\mathbf{T}_1}{z}+\frac{\mathbf{T}_2}{z^2}+O\left(\frac{1}{z^3}\right), \qquad \textrm{as } z \to \infty.
    \end{equation}}
    \item{\emph{Asymptotics at the origin}: $\mathbf{T}(z)$ satisfies the same behaviour near $0$ as $\mathbf{Y}(z)$.}
\end{itemize}

\subsection{Second transformation}
\label{sec:dl:rhp:2nd}

With the properties of $g(z)$ given in Proposition \ref{prop-g-func}, one can see that the diagonal entries of the jump matrix $J_{\mathbf{T}}(z)$  are highly oscillatory on $(0,b)$ for large $n$. To remove the oscillations, we will apply a contour deformation in the second transformation.  For $x \in (a,b)$, based on the following factorization
\begin{align*}
    J_ {\mathbf{T}}(x) &= \begin{pmatrix}
                e^{n(g_{-}(x) - g_{ +}(x))} & x^{\alpha-\frac{1}{2}}\\
                0 & e^{n(g_{+}(x) - g_{-}(x))}
            \end{pmatrix} \\
            &= \begin{pmatrix}
            1 & 0\\
            x^{-\alpha + \frac{1}{2}}e^{-n (2 g_{ -}(x) -cx -l)} & 1
            \end{pmatrix}
            \begin{pmatrix}
            0 & x^{\alpha - \frac{1}{2}}\\
            -x^{-\alpha + \frac{1}{2}} & 0
            \end{pmatrix}
            \begin{pmatrix}
            1 & 0\\
            x^{-\alpha + \frac{1}{2}}e^{-n (2 g_{ +}(x) -cx -l)} & 1
            \end{pmatrix},
\end{align*}
the well-known ``opening of the lens" will be conducted. Regarding the left part $(0,a)$, the jump $J_{\mathbf{T}}(x)$ is exponentially close to the diagonal matrix $e^{2 i n \pi \sqrt{x} \sigma_3}$. Based on these observations, we introduce the second transformation as follows:
\begin{equation}\label{ttos}
    \mathbf{S}(z) = \begin{cases}
            \mathbf{T}(z)B(z)^{-1}, & \textrm{$z \in \widetilde{\Omega} _+^{\bigtriangledown}$,}\\
            \mathbf{T}(z)B(z), & \textrm{$z \in \widetilde{\Omega} _-^{\bigtriangledown}$,}\\
            \mathbf{T}(z)A_{\pm}(z), & \textrm{$z \in \Omega _{\pm}^{\bigtriangleup}$,}\\
            \mathbf{T}(z), & \textrm{otherwise,}
            \end{cases}
\end{equation}
where the regions $\widetilde{\Omega} _{\pm}^{\bigtriangledown}$ and $\Omega _{\pm}^{\bigtriangleup}$ are depicted in Figure \ref{sigmaS}. In the above formula, the function $B(z)$ and $A_{\pm}(z)$ are defined as
\begin{equation}\label{J-def}
    B(z) = \begin{pmatrix}
    1 & 0\\
    z^{-\alpha + \frac{1}{2}}e^{-n (2 g(z) -cz -l)} & 1
    \end{pmatrix}
\end{equation}
and
\begin{equation}\label{A-def}
    A_{\pm}(z)=\begin{pmatrix}
    \mp \frac{e^{\mp i n \pi \sqrt{z}}}{i n \pi} & 0\\
    0 & \mp i n \pi e^{\pm i n \pi \sqrt{z}}
    \end{pmatrix}.
\end{equation}
\begin{figure}[h]
    \centering
    \includegraphics[width=0.6\textwidth]{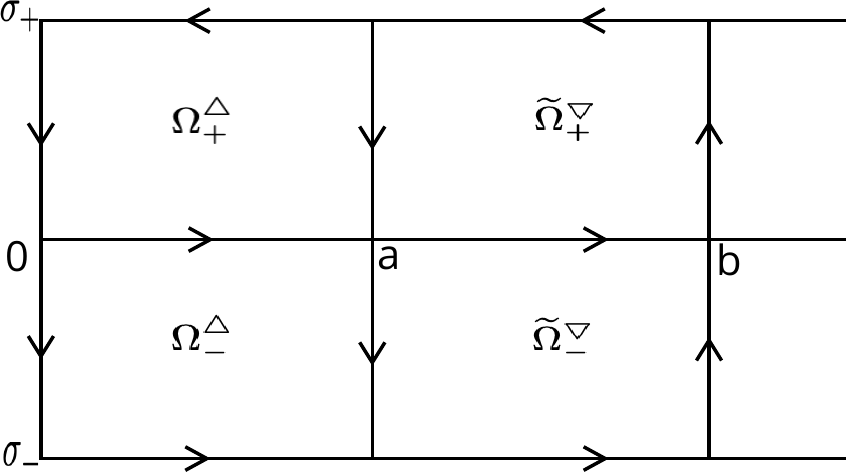}
    \caption{Contour $\Sigma_{\mathbf{S}}$ and regions $\widetilde{\Omega} _{\pm}^{\bigtriangledown}, \Omega _{\pm}^{\bigtriangleup}$.}
    \label{sigmaS}
\end{figure}

By a straightforward computation, one can see that  $\mathbf{S}(z)$ satisfies the following RH problem:
\begin{itemize}
    \item {\emph{Analyticity}: $\mathbf{S}(z)$ is analytic for $z \in \mathbb{C} \setminus \Sigma_{\mathbf{S}}$.}
    \item{\emph{Jump condition}: $\mathbf{S}_{+}(z)=\mathbf{S}_{-}(z)J_{\mathbf{S}}(z)$ for $z \in \Sigma_{\mathbf{S}}$,
    \begin{equation}
        J_{\mathbf{S}}(z)=
        \begin{cases}
            \begin{pmatrix}
                0 & x^{\alpha - \frac{1}{2}}\\
                -x^{-\alpha + \frac{1}{2}} & 0
            \end{pmatrix}, & x \in (a, b), \\
            \begin{pmatrix}
                1 & x^{\alpha-\frac{1}{2}}e^{n(g_{+}(x) + g_{-}(x) - cx -l)}\\
                0 & 1
            \end{pmatrix}, & x \in (b, \infty), \\
            \begin{pmatrix}
                -1 & 0\\
                - x^{-\alpha + \frac{1}{2}}e^{-n(g_{+}(x) + g_{-}(x) - cx -l)} & -1
            \end{pmatrix}, & x \in (0,a), \\
            \begin{pmatrix}
                \frac{1}{1- e ^{\pm 2 i n \pi \sqrt{z}}} & 0\\
                \mp z^{-\alpha + \frac{1}{2}}e^{-n (2 g(z) -cz -l)} & 1- e ^{\pm 2 i n \pi \sqrt{z}}
            \end{pmatrix}, & z \in \sigma_{\pm}, \\
            \begin{pmatrix}
                \frac{1}{1- e ^{\pm 2 i n \pi \sqrt{z}}} & \pm z^{\alpha - \frac{1}{2}} \frac{e^{n (2 g(z) -cz -l)}}{1- e ^{\mp 2 i n \pi \sqrt{z}}} \\
                \mp z^{-\alpha + \frac{1}{2}}e^{-n (2 g(z) -cz -l)} & 1
            \end{pmatrix}, & z \in (a, b) \pm i \varepsilon, \\
            \begin{pmatrix}
                1 & \pm z^{\alpha-\frac{1}{2}} \frac{e^{n(2g(z) - cz -l)}}{1-e^{\mp 2 i n \pi \sqrt{z}}}\\
                0 & 1
            \end{pmatrix}, & z \in (b, \infty) \pm i \varepsilon, \\
            \begin{pmatrix}
                1 & 0\\
                \mp z^{-\alpha + \frac{1}{2}}e^{-n (2 g(z) -cz -l)} & 1
            \end{pmatrix}, & z \in (b, b \pm i \varepsilon), \\
            \begin{pmatrix}
                1 & \mp z^{\alpha - \frac{1}{2}} e^{n (2 g(z) -cz -l)) \pm i n \pi \sqrt{z}}\\
                0 & 1
            \end{pmatrix}, & z \in (a, a \pm i \varepsilon).
        \end{cases}
    \end{equation}
    }
    \item{\emph{Asymptotics at infinity}:
    \begin{equation}
        \mathbf{S}(z) = I + \frac{\mathbf{S}_1}{z}+\frac{\mathbf{S}_2}{z^2}+O\left(\frac{1}{z^3}\right), \qquad \textrm{as } z \to \infty.
    \end{equation}}
    \item{\emph{Asymptotics at the origin}: $\mathbf{S}(z)$ satisfies the same behaviour near 0 as $\mathbf{Y}(z)$.}
\end{itemize}

\subsection{Global parametrix}
\label{sec:dl:gp}

When $n$ is large, due to the properties of the $g$-function in Proposition \ref{prop-g-func}, the jump matrix $J_{\mathbf{S}}(z)$ tends to the identity matrix for $z$ bounded away from the interval $(0, b)$. Therefore, we consider the following global parametrix.
\begin{itemize}
    \item {\emph{Analyticity}: $\mathbf{N}(z)$ is analytic for $z \in \mathbb{C}\setminus [0,b]$}.
    \item {\emph{Jump condition}: \begin{equation} \label{N-jump}
            \mathbf{N}_{+}(x)=\mathbf{N}_{-}(x)\begin{cases}
            \begin{pmatrix}
                0 & x^{\alpha - \frac{1}{2}}\\
                -x^{-\alpha + \frac{1}{2}} & 0
            \end{pmatrix}, & x \in (a, b) , \\
            \begin{pmatrix}
                -1 & 0\\
                0 & -1
            \end{pmatrix}, & x \in (0,a).
    \end{cases}
    \end{equation}
    }
    \item {\emph{Asymptotics at infinity}:
    \begin{equation} \label{Mforzlarge}
        \mathbf{N}(z) = I + \frac{\mathbf{N}_1}{z}+\frac{\mathbf{N}_2}{z^2}+O\left(\frac{1}{z^3}\right) , \qquad \textrm{as } z \to \infty.
    \end{equation}}
\end{itemize}

The solution to the above RH problem is given explicitly as
\begin{equation}\label{global}
    \mathbf{N}(z)=D_{\infty}^{\sigma_3} \begin{pmatrix}
    \frac{\gamma(z)+\gamma(z)^{-1}}{2}  & \frac{\gamma(z)-\gamma(z)^{-1}}{-2i} \\
    \frac{\gamma(z)-\gamma(z)^{-1}}{2i}  & \frac{\gamma(z)+\gamma(z)^{-1}}{2}
    \end{pmatrix}D(z)^{-\sigma_3},
\end{equation}
where $D(z)$ is the Szeg\H{o} function given in \eqref{szego}, the function $
\gamma(z)$ is defined in \eqref{gamma-def}, and $D_{\infty}$ is the constant in \eqref{Dinfty-def}. Denote
\begin{equation}\label{scalar}
     \widetilde{\mathbf{N}}(z)=  \begin{pmatrix}
    \frac{\gamma(z)+\gamma(z)^{-1}}{2}  & \frac{\gamma(z)-\gamma(z)^{-1}}{-2i} \\
    \frac{\gamma(z)-\gamma(z)^{-1}}{2i}  & \frac{\gamma(z)+\gamma(z)^{-1}}{2}
    \end{pmatrix},
\end{equation}
it is easy to check that $\widetilde{\mathbf{N}}(z)$ satisfies a RH problem as follows:
\begin{itemize}
    \item {\emph{Analyticity}: $\widetilde{\mathbf{N}}(z)$ is analytic for $z \in \mathbb{C}\setminus [0,b]$}.
    \item {\emph{Jump condition}:
    \begin{equation}
            \widetilde{\mathbf{N}}_+(z)=\widetilde{\mathbf{N}}_-(z)\begin{cases}
            \begin{pmatrix}
                0 & 1\\
                -1 & 0
            \end{pmatrix}, & x \in (a, b) ,\\
            \begin{pmatrix}
                -1 & 0\\
                0 & -1
            \end{pmatrix}, & x \in (0,a).
    \end{cases}
    \end{equation}
    }
    \item {\emph{Asymptotics at infinity}: $
        \widetilde{\mathbf{N}}(z) = I +O\left(\frac{1}{z}\right)$ as $z \to \infty$.
    }
\end{itemize}
Using the relation $D_+(x)D_-(x)=x^{\alpha - \frac{1}{2}}$ for $x \in (a, b)$, one immediately sees that $ \mathbf{N}(z)$ defined in \eqref{global} solves the RH problem for $ \mathbf{N}.$

%

\begin{remark}
    Obviously, $\det \mathbf{N}(z) \equiv 1$.  In addition,  we have the following behaviour near the endpoints:
        \begin{equation}\label{Nbehaviour}
        \mathbf{N}(z) = \begin{cases}
            \begin{pmatrix}
                O(|z|^{-\frac{1}{2}}) & O(|z|^{-\frac{1}{2}}) \\
                O(|z|^{-\frac{1}{2}}) & O(|z|^{-\frac{1}{2}})
            \end{pmatrix}, & \textrm{as $z \to 0$,}\\
            \begin{pmatrix}
                O(|z-a|^{-\frac{1}{4}}) & O(|z-a|^{-\frac{1}{4}}) \\
                O(|z-a|^{-\frac{1}{4}}) & O(|z-a|^{-\frac{1}{4}})
            \end{pmatrix}, & \textrm{as $z \to a$,}\\
            \begin{pmatrix}
                O(|z-b|^{-\frac{1}{4}}) & O(|z-b|^{-\frac{1}{4}}) \\
                O(|z-b|^{-\frac{1}{4}}) & O(|z-b|^{-\frac{1}{4}})
            \end{pmatrix}, & \textrm{as $z \to b$.}
        \end{cases}
    \end{equation}
\end{remark}

For later use, let us derive the asymptotics of $\mathbf{N}(z)$ as $z \to \infty$ in more details. From  the definitions of $D(z)$ and  $\gamma(z)$ in \eqref{szego} and \eqref{scalar}, respectively,  straightforward computations give us
\begin{equation}
    D(z) = D_{\infty} \left(1 + \frac{(2\alpha -1)(a+b-2\sqrt{a b})}{8z} + \frac{\Delta}{z^2} + O\left(\frac{1}{z^3}\right)\right), \qquad z \to \infty,
\end{equation}
with $\Delta = - \frac{(2\alpha-1)(-5a^2 +2a b -5 b^2 +4 a \sqrt{a b} +4 b \sqrt{a b} -2\alpha a^2 -12\alpha a b -2\alpha b^2 + 8\alpha a \sqrt{a b} + 8\alpha b\sqrt{a b})}{128}$ and
\begin{equation}
    \gamma(z) = 1 + \frac{a+b}{4z} + \frac{5 a^2 + 2 a b + 5b^2}{32z^2} + O\left(\frac{1}{z^3}\right), \qquad z \to \infty.
\end{equation}
Substituting the above approximations into \eqref{global},  we get the large-$z$ asymptotics \eqref{Mforzlarge}
with the coefficients given by
\begin{equation}\label{N1-def}
    \mathbf{N}_1 = D_{\infty}^{\sigma_3} \begin{pmatrix}
        -\frac{(2\alpha-1)(\sqrt{b}-\sqrt{a})^2}{8} & \frac{i (a+b)}{4}\\
        -\frac{i (a+b)}{4} & \frac{(2\alpha-1)(\sqrt{b}-\sqrt{a})^2}{8}
    \end{pmatrix} D_{\infty}^{-\sigma_3}
\end{equation}
and
\begin{equation}\label{N2-def}
    \mathbf{N}_2 = D_{\infty}^{\sigma_3} \begin{pmatrix}
        * & \frac{i (2\alpha-1)(\sqrt{b}-\sqrt{a})^2(a+b)+4(a^2+b^2)}{32}\\
        \frac{i (2\alpha-1)(\sqrt{b}-\sqrt{a})^2(a+b)-4(a^2+b^2)}{32} & *
    \end{pmatrix} D_{\infty}^{-\sigma_3},
\end{equation}
where * stands for certain unimportant entries.

\subsection{Local parametrix near the origin}
\label{sec:dl:lp:0}

Since jump matrices for $\mathbf{S}(z)$ and $\mathbf{N}(z)$ are not uniformly close to each other when $z$ is close to the endpoints 0, $a$ and $b$, local  parametrices  will be constructed  near these endpoints to approximate $\mathbf{S}$ for large $n$.

We start from the local parametrix near the origin. Let $D(0, \varepsilon):=  \{ z: |z| <  \varepsilon\}  $  be a small disk centered at the origin. We seek a $2 \times 2$ matrix-valued function $\mathbf{P}^{(0)}(z)$ satisfying a RH problem as follows:
\begin{itemize}
    \item {\emph{Analyticity}: $\mathbf{P}^{(0)}(z)$ is analytic for $z \in D(0, \varepsilon) \setminus \Sigma_{\mathbf{S}}$.}
    \item{\emph{Jump condition}: $\mathbf{P}^{(0)}_{+}(z)=\mathbf{P}^{(0)}_{-}(z)J_{\mathbf{P}^{(0)}}(z)$ for $z \in D(0, \varepsilon) \cap \Sigma_{\mathbf{S}}$,
    \begin{equation}
        J_{\mathbf{P}^{(0)}}(z)=
        \begin{cases}
            \begin{pmatrix}
                -1 & 0\\
                0 & -1
            \end{pmatrix}, & z \in (0, \varepsilon), \\
            \begin{pmatrix}
                \frac{1}{1- e ^{\pm 2 i n \pi \sqrt{z}}} & 0\\
                0 & 1- e ^{\pm 2 i n \pi \sqrt{z}}
            \end{pmatrix}, & z \in (0, \pm i \varepsilon).
        \end{cases}
    \end{equation}
    }
    \item{\emph{Matching condition}: for $z \in \partial D(0, \varepsilon)$,
    \begin{equation} \label{match0}
        \mathbf{P}^{(0)}(z) = \left(I+ O\left( \frac{1}{n} \right)\right) \mathbf{N}(z), \qquad \textrm{as } n \to \infty.
    \end{equation}}
\end{itemize}

The solution to the above RH problem is given in terms of Gamma functions; see the similar constructions in \cite{lin2013}. Recall the definitions of $H^*(z)$ and $H(z)$  in \eqref{H*-def} and \eqref{H-def}. Using the Stirling's formula, one can see that both $H(z)$ and $H^*(z)$ tend to $1$ as $n \to \infty$.
%
Then, the function
\begin{equation}
    Q(z) = \begin{cases}
            H(z)^{\sigma_3}, & \re z >0\\
            H^*(z)^{\sigma_3}, & \re z<0
    \end{cases}
\end{equation}
satisfies a RH problem as follows:
\begin{itemize}
    \item {\emph{Analyticity}: $Q(z)$ is analytic for $z \in \mathbb{C} \setminus (- i \infty, i \infty)$.}
    \item{\emph{Jump condition}: for $z \in (- i \infty, i \infty)$,
    \begin{equation}
        Q_+(z)=Q_-(z)
        \begin{cases}
            \begin{pmatrix}
                \frac{1}{1- e ^{ 2 i n \pi \sqrt{z}}} & 0\\
                0 & 1- e ^{ 2 i n \pi \sqrt{z}}
            \end{pmatrix}, & z \in (0,  i \infty), \\
            \begin{pmatrix}
                \frac{1}{1- e ^{- 2 i n \pi \sqrt{z}}} & 0\\
                0 & 1- e ^{- 2 i n \pi \sqrt{z}}
            \end{pmatrix}, & z \in (- i \infty, 0).
        \end{cases}
    \end{equation}
    }
    \item{\emph{Asymptotics at infinity}:
    \begin{equation}
        Q(z) = I+ O\left( \frac{1}{n} \right), \qquad \textrm{as } n \to \infty.
    \end{equation}}
\end{itemize}
From the above RH problem, it is easy to get the following solution to the RH problem for $\mathbf{P}^{(0)}(z)$:
\begin{equation}\label{P0-def}
    \mathbf{P}^{(0)}(z) = \mathbf{N}(z) Q(z) = \mathbf{N}(z) \begin{cases}
            H(z)^{\sigma_3}, & \re z >0,\\
            H^*(z)^{\sigma_3}, & \re z<0.
    \end{cases}
\end{equation}

\subsection{Local parametrix near $b$}

Next we move to the ``band-void" endpoint $b$. Let $D(b, \varepsilon):=  \{ z: |z-b| <  \varepsilon\}$ be a small disk centered at $b$. We look for a $2 \times 2$ matrix-valued function $\mathbf{P}^{(b)}(z)$ in $D(b, \varepsilon)$ such that it satisfies the following RH problem:
\begin{itemize}
    \item {\emph{Analyticity}: $\mathbf{P}^{(b)}(z)$ is analytic for $z \in \mathbb{C} \setminus D(b, \varepsilon)$.}
    \item{\emph{Jump condition}: for  $z \in D(b, \varepsilon) \cap \Sigma_{\mathbf{S}}$,
    \begin{equation}\label{jumppb}
        \mathbf{P}^{(b)}_+(z)=\mathbf{P}^{(b)}_-(z) J_{\mathbf{S}}(z).
    \end{equation}
    }
    \item{\emph{Matching condition}: for $z \in \partial D(b, \varepsilon)$,
    \begin{equation}\label{matchb}
        \mathbf{P}^{(b)}(z)(z) = \left(I+ O\left( \frac{1}{n} \right)\right)\mathbf{N}(z), \qquad \textrm{as } n \to \infty.
    \end{equation}}
\end{itemize}

The local parametrix is constructed in terms of the Airy function. Let $f(z)$ be the function defined in \eqref{f-def}, which is analytic in $D(b, \varepsilon) $, the solution is given by
\begin{equation}\label{Pb-def}
    \mathbf{P}^{(b)}(z)
        =  E^{(b)}(z) \mathbf{\Psi} (n^{\frac{2}{3}} f(z)) e^{-n(g(z) - \frac{cz}{2} - \frac{l}{2})\sigma_3}z^{(-\frac{\alpha}{2}+\frac{1}{4})\sigma_3},
\end{equation}
where the pre-factor $ E^{(b)}(z)$ is defined as
\begin{equation} \label{lp:b:E}
E^{(b)}(z) = \mathbf{N}(z) z^{(\frac{\alpha}{2} - \frac{1}{4}) \sigma_3} e^{\frac{\pi i}{4}\sigma_3} \frac{1}{\sqrt{2}}  \begin{pmatrix}
    1 & -1\\
    1 & 1
    \end{pmatrix} n^{\frac{1}{6} \sigma_3} f(z)^{\frac{1}{4} \sigma_3}.
\end{equation}
The function $\mathbf{\Psi}(\cdot)$ is the well-known Airy model parametrix (cf. \cite{deift1999-2}):
\begin{equation} \label{airy-pmx}
    \mathbf{\Psi}(\zeta) = \sqrt{2 \pi} e^{-\frac{\pi i}{12}}\times \begin{cases}
        \begin{pmatrix}
            \Ai (\zeta) & \Ai (\omega^2 \zeta)\\
            \Ai' (\zeta) & \omega^2 \Ai' (\omega^2 \zeta)
        \end{pmatrix}e^{-\frac{\pi i}{6}\sigma_3}, & \textrm{for $\zeta \in$ I},\\
        \begin{pmatrix}
            \Ai (\zeta) & \Ai (\omega^2 \zeta)\\
            \Ai' (\zeta) & \omega^2 \Ai' (\omega^2 \zeta)
        \end{pmatrix}e^{-\frac{\pi i}{6}\sigma_3}\begin{pmatrix}
            1 & 0\\
            -1 & 1
        \end{pmatrix}, & \textrm{for $\zeta \in$ II},\\
        \begin{pmatrix}
            \Ai (\zeta) & -\omega^2 \Ai (\omega \zeta)\\
            \Ai'(\zeta) & -\Ai' (\omega \zeta)
        \end{pmatrix}e^{-\frac{\pi i}{6}\sigma_3}\begin{pmatrix}
            1 & 0\\
            1 & 1
        \end{pmatrix}, & \textrm{for $\zeta \in$ III},\\
        \begin{pmatrix}
            \Ai (\zeta) & -\omega^2 \Ai (\omega \zeta)\\
            \Ai'(\zeta) & -\Ai' (\omega \zeta)
        \end{pmatrix}e^{-\frac{\pi i}{6}\sigma_3}, & \textrm{for $\zeta \in$ IV},
    \end{cases}
\end{equation}
with $\omega = e^{ \frac{2\pi i}{3} }$, where the regions I-IV are depicted in Figure \ref{sigma'}.
\begin{figure}[h]
    \centering
    \includegraphics[width=0.65\textwidth]{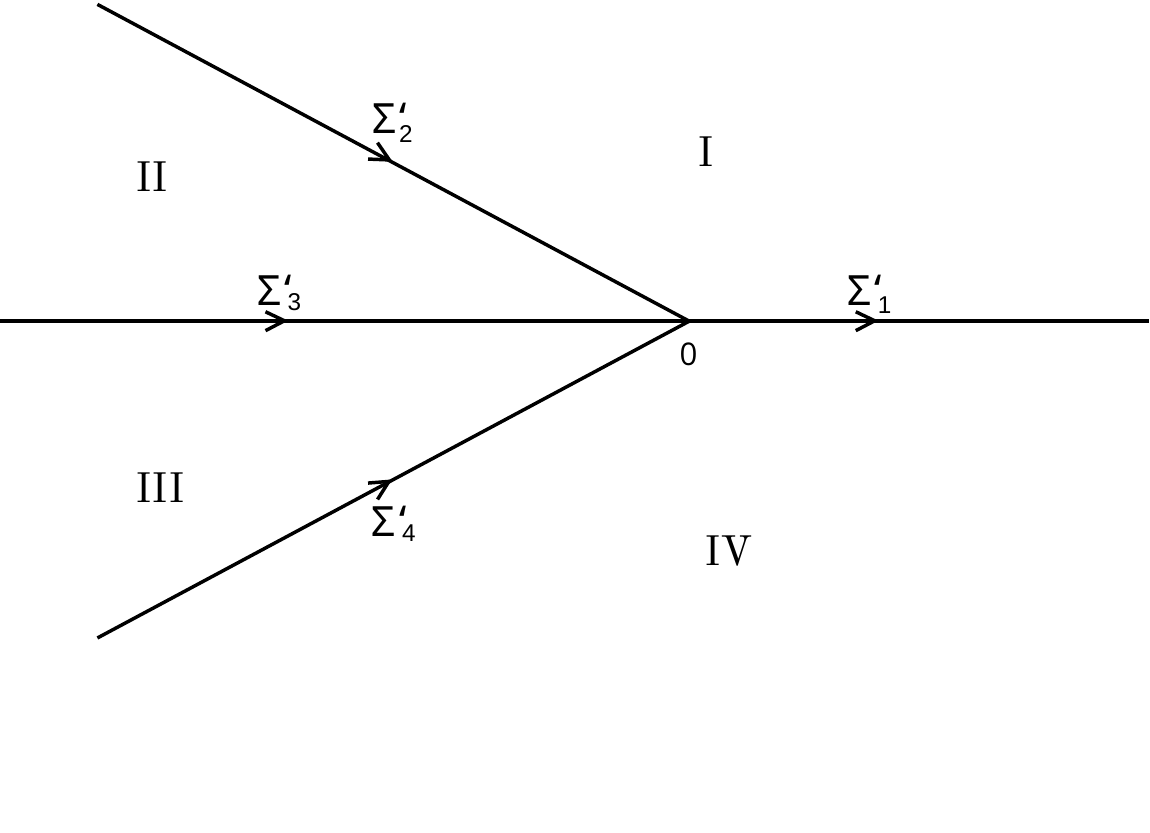} \vspace{-1.8cm}
    \caption{Regions and contours for $\mathbf{\Psi}$.}
    \label{sigma'}
\end{figure}

\begin{proposition} \label{prop-pb}
    The function defined in \eqref{Pb-def} satisfies the RH problem for $ \mathbf{P}^{(b)}$.
\end{proposition}
\begin{proof}

From \eqref{rho-b} and the definition of $f(z)$ in \eqref{f-def}, it is easy to see that
\begin{equation}
    f(z) = (C_2 \pi)^{\frac{2}{3}} (z-b) + O(z-b)^2,
\end{equation}
with $C_2 >0$. This implies that  $f(z)$ is analytic near $b$ and maps $D(b, \varepsilon)$ onto a convex neighbourhood of $0$. Then, it is easy to verify that the function defined in \eqref{Pb-def} satisfies the jump conditions $J_{\mathbf{S}}(z)$ for $z \in D(b, \varepsilon) \cap \Sigma_{\mathbf{S}}$ provided that the pre-factor $E^{(b)}(z)$  is analytic in $D(b, \varepsilon)$.


From the definition in \eqref{lp:b:E}, it is obvious that $E^{(b)}(z)$ is analytic in $D(b, \varepsilon) \setminus (-\infty, b)$. For $z \in (b-\varepsilon, b)$,  recalling the jump condition \eqref{N-jump} of the RH problem for $\mathbf{N}(z)$ and the fact that $f_{ +}(x)^{\frac{\sigma_3}{4}}=f_{ -}(x)^{\frac{\sigma_3}{4}}e^{\frac{\pi i}{2}\sigma_3}$ on $(b- \varepsilon, b)$, we have
    \begin{equation}
        \begin{split}
            E^{(b)}_+(x) & = \mathbf{N}_+(x) x^{(\frac{\alpha}{2} - \frac{1}{4}) \sigma_3} e^{\frac{\pi i}{4}\sigma_3} \frac{1}{\sqrt{2}}  \begin{pmatrix}
    1 & -1\\
    1 & 1
    \end{pmatrix} n^{\frac{1}{6} \sigma_3} f_{+}(x)^{\frac{1}{4} \sigma_3}\\
        & = \mathbf{N}_-(x) \begin{pmatrix}
                0 & x^{\alpha - \frac{1}{2}}\\
                -x^{-\alpha + \frac{1}{2}} & 0
            \end{pmatrix} x^{(\frac{\alpha}{2} - \frac{1}{4}) \sigma_3} e^{\frac{\pi i}{4}\sigma_3} \frac{1}{\sqrt{2}}  \begin{pmatrix}
    1 & -1\\
    1 & 1
    \end{pmatrix} n^{\frac{1}{6} \sigma_3} f_{-}(x)^{\frac{1}{4} \sigma_3} e^{\frac{\pi i}{2}\sigma_3}\\
    & = \mathbf{N}_-(x) x^{(\frac{a}{2} - \frac{1}{4}) \sigma_3} e^{\frac{\pi i}{4}\sigma_3} \frac{1}{\sqrt{2}}  \begin{pmatrix}
    1 & -1\\
    1 & 1
    \end{pmatrix} n^{\frac{1}{6} \sigma_3} f_{-}(x)^{\frac{1}{4} \sigma_3}\\
    &=E^{(b)}_-(x).
        \end{split}
    \end{equation}
 This means $E^{(b)}(z)$ has no jump in $D(b, \varepsilon)$. Moreover, from behaviour of $\mathbf{N}(z)$ at $b$ in \eqref{Nbehaviour}, one can see that $E^{(b)}(z)$ has at most square-root singularities at $b$. Therefore, $z = b$ is a removable singularity, and  $E^{(b)}(z)$ is indeed analytic in $D(b, \varepsilon)$.


Finally, let us check the matching condition on $\partial D(b, \varepsilon)$. Recalling the following behaviour of $ \mathbf{\Psi}(\zeta)$ as $\zeta \to \infty$:
 \begin{equation}\label{psiinfty}
        \mathbf{\Psi}(\zeta)= \zeta^{-\frac{\sigma_3}{4}}\frac{1}{\sqrt{2}}
        \begin{pmatrix}
            1 & 1\\
            -1 & 1
        \end{pmatrix}\left[I+ \frac{1}{48 \zeta^{\frac{3}{2}}}\begin{pmatrix}
            1 & 6\\
            -6 & -1
        \end{pmatrix} + O\left( \frac{1}{\zeta^3} \right)\right]e^{-\frac{\pi i}{4}\sigma_3}e^{-\frac{2}{3}\zeta^{\frac{3}{2}}\sigma_3}.
    \end{equation}
Since $\zeta = n^{\frac{2}{3}} f(z) $ is large uniformly for $z \in \partial D(b, \varepsilon)$ when $n \to \infty$, we have from the above formula and \eqref{Pb-def} that
\begin{eqnarray}
 \mathbf{P}^{(b)}(z)
        &= & \mathbf{N}(z) z^{(\frac{\alpha}{2} - \frac{1}{4}) \sigma_3} e^{\frac{\pi i}{4}\sigma_3} \frac{1}{\sqrt{2}}  \begin{pmatrix}
    1 & -1\\
    1 & 1
    \end{pmatrix} n^{\frac{1}{6} \sigma_3} f(z)^{\frac{1}{4} \sigma_3} \nonumber  \\
        && \times (n^{\frac{2}{3}} f(z))^{-\frac{\sigma_3}{4}}\frac{1}{\sqrt{2}}
        \begin{pmatrix}
            1 & 1\\
            -1 & 1
        \end{pmatrix}\left[I+ \frac{1}{48 n f(z)^{\frac{3}{2}}}\begin{pmatrix}
            1 & 6\\
            -6 & -1
        \end{pmatrix} + O\left( \frac{1}{n^2} \right)\right] \nonumber \\
       & & \times e^{-\frac{\pi i}{4}\sigma_3}e^{-\frac{2}{3} n f(z)^{\frac{3}{2}}\sigma_3} e^{-n(g(z) - \frac{cz}{2} - \frac{l}{2})\sigma_3}z^{(-\frac{\alpha}{2}+\frac{1}{4})\sigma_3} \nonumber \\
       & = & \mathbf{N}(z) \left[ I+ \frac{1}{48 n f(z)^{\frac{3}{2}}} \begin{pmatrix}
        1 & 6i z^{\alpha -\frac{1}{2}}\\
        6i z^{-\alpha +\frac{1}{2}} & -1
        \end{pmatrix} + O\left( \frac{1}{n^2} \right) \right], \quad n \to \infty.
\end{eqnarray}
This completes the proof of the proposition.
\end{proof}

\subsection{Local parametrix near $a$}

Next, we move to the ``saturated region-band" endpoint $a$ and consider a small disk $D(a, \varepsilon): = \{ z: |z-a| <  \varepsilon\}$. We seek a local parametrix $\mathbf{P}^{(a)}(z)$ defined on $D(a, \varepsilon)$ such that
\begin{itemize}
    \item {\emph{Analyticity}: $\mathbf{P}^{(a)}(z)$ is analytic for $z \in \mathbb{C} \setminus D(a, \varepsilon)$.}
    \item{\emph{Jump condition}:for $z \in D(a, \varepsilon) \cap \Sigma_{\mathbf{S}}$,
    \begin{equation}\label{jumppa}
        \mathbf{P}^{(a)}_+(z)=\mathbf{P}^{(a)}_-(z)J_{\mathbf{S}}(z).
    \end{equation}
    }
\item{\emph{Matching condition}: for $z \in \partial D(a, \varepsilon)$,
    \begin{equation}\label{matcha}
        \mathbf{P}^{(a)}(z) = \left(I+ O\left( \frac{1}{n} \right)\right)\mathbf{N}(z), \qquad \textrm{as } n \to \infty.
    \end{equation}}
\end{itemize}

The parametrix construction is similar to that in $D(b, \varepsilon)$, which is also given in terms of the Airy functions. With the function $\widetilde{f}(z)$ defined in \eqref{f2-def}, we have
\begin{equation} \label{Pa-def}
     \mathbf{P}^{(a)}(z)
        = \pm E^{(a)}(z) \widetilde{\Psi} (n^{\frac{2}{3}} \widetilde{f}(z))\sigma_1 e^{-n(g(z) - \frac{cz}{2} - \frac{l}{2} \pm i  \pi \sqrt{z})\sigma_3}z^{(-\frac{\alpha}{2}+\frac{1}{4})\sigma_3} \quad \textrm{for } \pm \im{z} \ge 0,
\end{equation}
where
\begin{equation}\label{lp:a:E}
    E^{(a)}(z) = \pm (-1)^n \mathbf{N}(z) z^{(\frac{\alpha}{2} - \frac{1}{4})\sigma_3} e^{-\frac{\pi i}{4}\sigma_3}\frac{1}{\sqrt{2}}\begin{pmatrix}
    -1 & 1\\
    1 & 1
    \end{pmatrix}n^{\frac{1}{6} \sigma_3} (-\widetilde{f}(z))^{\frac{1}{4} \sigma_3} \quad \textrm{for } \pm \im{z} \ge 0,
\end{equation}
and
\begin{equation}
    \widetilde{\Psi}(\zeta) = \sigma_3 \Psi (-\zeta) \sigma_3.
\end{equation}

\begin{proposition}
    The function defined in \eqref{Pa-def} satisfies the RH problem for $ \mathbf{P}^{(a)}$.
\end{proposition}
\begin{proof}
The proof is similar to that of Proposition \ref{prop-pb}. The main task is to check the pre-factor $E^{(a)}(z)$ is analytic in $D(a, \varepsilon)$. According its definition in \eqref{lp:a:E}, we need to verify that $E^{(a)}(z)$ has no jumps on $(a-\varepsilon, a+\varepsilon)$ and $z = a$ is a removable singularity.

 For $x \in (a-\varepsilon, a)$, recall that $\mathbf{N}_+(x) = -\mathbf{N}_-(x)$. Then, we have
    \begin{equation}
        \begin{split}
            E^{(a)}_+(x)
            = & (-1)^n \mathbf{N}_+(x) x^{(\frac{\alpha}{2} - \frac{1}{4})\sigma_3} e^{-\frac{\pi i}{4}\sigma_3}\frac{1}{\sqrt{2}}\begin{pmatrix}
    -1 & 1\\
    1 & 1
    \end{pmatrix}n^{\frac{1}{6} \sigma_3} (-\widetilde{f}(x))^{\frac{1}{4} \sigma_3}\\
            = & -(-1)^n \mathbf{N}_-(x) x^{(\frac{\alpha}{2} - \frac{1}{4})\sigma_3} e^{-\frac{\pi i}{4}\sigma_3}\frac{1}{\sqrt{2}}\begin{pmatrix}
    -1 & 1\\
    1 & 1
    \end{pmatrix}n^{\frac{1}{6} \sigma_3} (-\widetilde{f}(x))^{\frac{1}{4} \sigma_3}\\
            = & E^{(a)}_-(x).
        \end{split}
    \end{equation}
  For $x \in (a, a+\varepsilon)$, we get from \eqref{N-jump} that
    \begin{eqnarray}
            E^{(a)}_+(x)
            &= & (-1)^n \mathbf{N}_+(x) x^{(\frac{\alpha}{2} - \frac{1}{4})\sigma_3} e^{-\frac{\pi i}{4}\sigma_3}\frac{1}{\sqrt{2}}\begin{pmatrix}
    -1 & 1\\
    1 & 1
    \end{pmatrix}n^{\frac{1}{6} \sigma_3} (-\widetilde{f}_{+}(x))^{\frac{1}{4} \sigma_3} \nonumber \\
            &= & (-1)^n \mathbf{N}_-(x) \begin{pmatrix}
            0 & x^{\alpha-\frac{1}{2}}\\
            -x^{-\alpha+\frac{1}{2}} & 0
            \end{pmatrix} x^{(\frac{\alpha}{2} - \frac{1}{4})\sigma_3} e^{-\frac{\pi i}{4}\sigma_3}\frac{1}{\sqrt{2}}\begin{pmatrix}
    -1 & 1\\
    1 & 1
    \end{pmatrix}n^{\frac{1}{6} \sigma_3} (-\widetilde{f}_-(x)e^{-2 \pi i})^{\frac{1}{4} \sigma_3} \nonumber  \\
            &= & -(-1)^n \mathbf{N}_-(x) x^{(\frac{\alpha}{2} - \frac{1}{4})\sigma_3} e^{-\frac{\pi i}{4}\sigma_3}\frac{1}{\sqrt{2}}\begin{pmatrix}
    -1 & 1\\
    1 & 1
    \end{pmatrix}n^{\frac{1}{6} \sigma_3} (-\widetilde{f}_-(x))^{\frac{1}{4} \sigma_3} \nonumber \\
            &= & E^{(a)}_-(x),
    \end{eqnarray}
    which means $E^{(a)}(z)$ has no jumps on $(a - \varepsilon, a + \varepsilon).$ Moreover, with behaviour of $\mathbf{N}(z)$ at $a$ in \eqref{Nbehaviour}, $E^{(a)}(z)$ has at most square-root singularities at $a$. Therefore, $z = a$  is a removable singularity of $E^{(a)}(z)$.

For the matching condition on $\partial D(a, \varepsilon)$, we have from \eqref{psiinfty} and  \eqref{Pa-def} that
\begin{eqnarray}
        \mathbf{P}^{(a)}(z)
        &= &  (-1)^n \mathbf{N}(z) z^{(\frac{\alpha}{2} - \frac{1}{4})\sigma_3} e^{-\frac{\pi i}{4}\sigma_3}\frac{1}{\sqrt{2}}\begin{pmatrix}
    -1 & 1\\
    1 & 1
    \end{pmatrix}n^{\frac{1}{6} \sigma_3} (-\widetilde{f}(z))^{\frac{1}{4} \sigma_3} \nonumber \\
       & & \times \sigma_3 (n^{\frac{2}{3}} (-\widetilde{f}(z)))^{-\frac{\sigma_3}{4}}\frac{1}{\sqrt{2}}
        \begin{pmatrix}
            1 & 1\\
            -1 & 1
        \end{pmatrix}\left[I+ \frac{1}{48 n (-\widetilde{f}(z))^{\frac{3}{2}}}\begin{pmatrix}
            1 & 6\\
            -6 & -1
        \end{pmatrix} + O\left( \frac{1}{n^2} \right)\right] \nonumber \\
        && \times e^{-\frac{\pi i}{4}\sigma_3}e^{-\frac{2}{3} n (-\widetilde{f}(z))^{\frac{3}{2}}\sigma_3} \sigma_3 \sigma_1 e^{-n(g(z) - \frac{cz}{2} - \frac{l}{2} \pm i  \pi \sqrt{z})\sigma_3}z^{(-\frac{\alpha}{2}+\frac{1}{4})\sigma_3} \nonumber \\
        &= & \mathbf{N}(z) \left[ I+ \frac{1}{48 n (-\widetilde{f}(z))^{\frac{3}{2}}} \begin{pmatrix}
        -1 & -6i z^{\alpha -\frac{1}{2}}\\
        -6i z^{-\alpha +\frac{1}{2}} & 1
        \end{pmatrix} + O\left( \frac{1}{n^2} \right) \right].
\end{eqnarray}

    This finishes the proof of the proposition.
\end{proof}

\subsection{Final transformation}
Now we perform the final transformation of our RH problem. Define
\begin{equation}\label{stor}
    \mathbf{R}(z) = \begin{cases}
    \mathbf{S}(z) \mathbf{P}^{(0)}(z)^{-1}, & \textrm{for $z \in D(0, \varepsilon) \setminus \Sigma_{\mathbf{S}}$,}\\
    \mathbf{S}(z) \mathbf{P}^{(a)}(z)^{-1}, & \textrm{for $z \in D(a, \varepsilon) \setminus \Sigma_{\mathbf{S}}$,}\\
    \mathbf{S}(z) \mathbf{P}^{(b)}(z)^{-1}, & \textrm{for $z \in D(b, \varepsilon) \setminus \Sigma_{\mathbf{S}}$,}\\
     \mathbf{S}(z) \mathbf{N}(z)^{-1}, & \textrm{otherwise}.
    \end{cases}
\end{equation}
Then, $\mathbf{R}(z)$ satisfies the following RH problem:
\begin{itemize}
    \item {\emph{Analyticity}: $\mathbf{R}(z)$ is analytic for $z \in \mathbb{C} \setminus \Sigma_{\mathbf{R}}$.}
    \begin{figure}[h]
    \centering
    \includegraphics[width=0.8\textwidth]{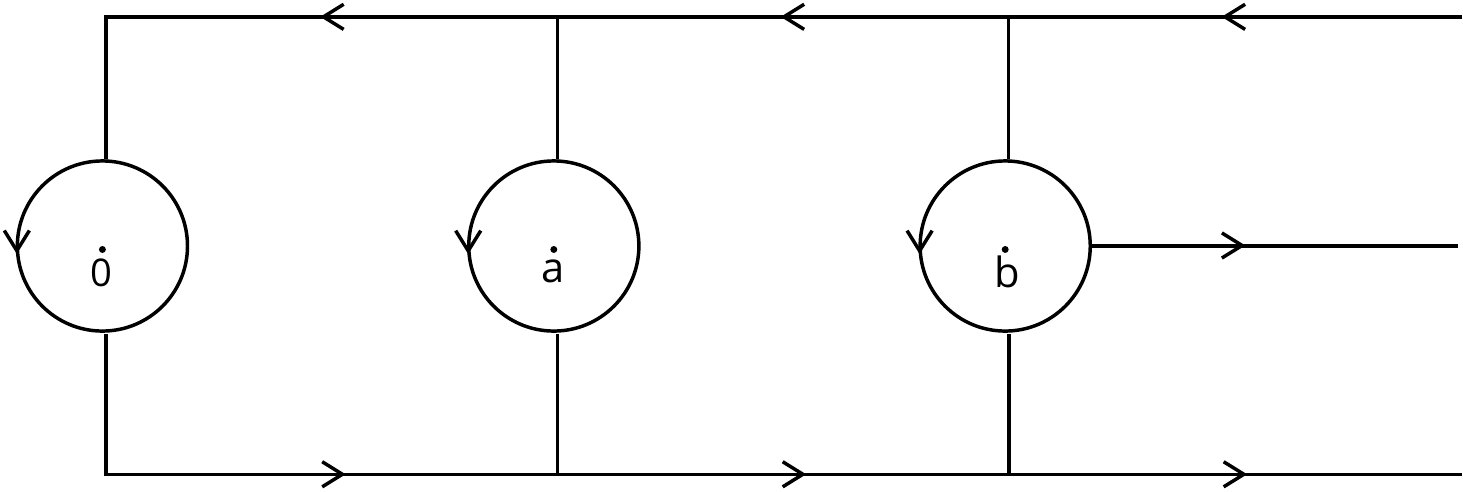}
    \caption{Contour $\Sigma_{\mathbf{R}}$.}
    \label{sigmaR}
\end{figure}

    \item{\emph{Jump condition}: $ \mathbf{R}_+(z) = \mathbf{R}_-(z) J_{\mathbf{R}}(z)$ for $z \in \Sigma_{\mathbf{R}}$,
    \begin{equation} \label{R-jump}
    J_{\mathbf{R}}(z) = \begin{cases}
        \mathbf{N}(z) \mathbf{P}^{(0)}(z)^{-1}, & \textrm{for $z \in \partial D(0, \varepsilon)$,}\\
        \mathbf{N}(z) \mathbf{P}^{(a)}(z)^{-1}, & \textrm{for $z \in \partial D(a, \varepsilon)$,}\\
        \mathbf{N}(z) \mathbf{P}^{(b)}(z)^{-1}, & \textrm{for $z \in \partial D(b, \varepsilon)$,}\\
        \mathbf{N}(z) J_{\mathbf{S}}(z) \mathbf{N}(z)^{-1}, & \textrm{elsewhere}.
    \end{cases}
\end{equation}
    }
    \item{\emph{Asymptotics at infinity}:
    \begin{equation} \label{R-large-z}
        \mathbf{R}(z) = I + \frac{\mathbf{R}_1}{z} + \frac{\mathbf{R}_2}{z^2} + O\left(\frac{1}{z^3}\right), \qquad \textrm{as } z \to \infty.
    \end{equation}}
\end{itemize}

Based on our local parametrix constructions near the endpoints 0, $a$ and $b$, $\mathbf{R}(z)$ has no jumps in $D(0, \varepsilon)$, $D(a, \varepsilon)$ and $D(b, \varepsilon)$, respectively. Then, the jump condition in \eqref{R-jump} is easily verified. We only need to check that $\mathbf{R}(z)$  has no poles  at $0, a$ and $b$.

For $z$ near $0$, since $\mathbf{S}(z)$ satisfies the same behaviour near 0 as $\mathbf{Y}(z)$, we have from \eqref{Y-origin}, \eqref{Nbehaviour} and  \eqref{P0-def} that
    \begin{equation}
        \mathbf{R}(z) =
        \begin{cases}
            \begin{pmatrix}
                O(|z|^{-\frac{3}{4}}) & O(|z|^{-\frac{3}{4}})\\
                O(|z|^{-\frac{3}{4}}) & O(|z|^{-\frac{3}{4}})
            \end{pmatrix}, & \re{z} \le 0,\\
            \begin{pmatrix}
                O(|z|^{-\frac{3}{4}}) & O(|z|^{-\frac{3}{4}})\\
                O(|z|^{-\frac{3}{4}}) & O(|z|^{-\frac{3}{4}})
            \end{pmatrix}, & \re{z} > 0, \alpha \le 1,\\
            \begin{pmatrix}
                O(|z|^{\frac{1}{4} - \alpha}) &  O(|z|^{\frac{1}{4} - \alpha})\\
                O(|z|^{\frac{1}{4} - \alpha}) &  O(|z|^{\frac{1}{4} - \alpha})
            \end{pmatrix}, & \re{z} > 0, \alpha > 1,
        \end{cases} \qquad \textrm{as } z \to 0,
    \end{equation}
which shows that $\mathbf{R}(z)$ has a removable singularity at $0$. For $z$ near the endpoints $a$ and $b$,  as the Airy parametrix $\mathbf{\Psi}(\zeta)$ in \eqref{airy-pmx} is bounded near $0$, $\mathbf{P}^{(b)}(z)$ and $\mathbf{P}^{(a)}(z)$ are also bounded near $b$ and $a$ according to their definitions in \eqref{Pb-def} and \eqref{Pa-def}. Therefore, $\mathbf{S}(z)$ remains bounded near $b$ and $a$, which means that $\mathbf{R}(z)$ is analytic at all of the three endpoints 0, $a$ and $b$.

%
%

From the matching conditions in \eqref{match0}, \eqref{matchb} and \eqref{matcha}, one can see the jump $J_{\mathbf{R}}(z)$ in \eqref{R-jump} has the following behaviour when  $n \to \infty$:
\begin{equation}
    J_{\mathbf{R}}(z) = \begin{cases}
        I + O\left( \frac{1}{n} \right), & \textrm{on $\partial D(0, \varepsilon) \cup \partial D(a, \varepsilon) \cup \partial D(b, \varepsilon)$},\\
        I + O(e^{-\delta n}), & \textrm{on the rest of $\Sigma_{\mathbf{R}}$},
    \end{cases}
\end{equation}
where $\delta > 0$ is some fixed constant and all the $O$-terms hold uniformly on their respective contours. Therefore, the RH problem for $\mathbf{R}(z)$ is a small-norm RH problem. By a standard argument in \cite{deift1999-1}, we get
    \begin{equation}\label{Rbehaviour}
        \mathbf{R}(z) = I + O\left(\frac{1}{n(|z|+1)}\right), \qquad \textrm{as } n \to \infty,
    \end{equation}
    uniformly for $z \in \mathbb{C} \setminus \Sigma_{\mathbf{R}}$.

\section{Proof of the main results}
\label{sec:dl:proof}

The steepest descent analysis for RH problems has been completed.  Then, we are ready to prove our main results.


\subsection{Proof of Theorem \ref{coefficient}}

To get the asymptotics of $h_{n, n}$, $\mathscr{A}_{n, n}^2$ and $\mathscr{B}_{n, n}$, we need the large-$z$ asymptotics of $\mathbf{P}(z)$ and make use of the relations in  \eqref{hN} and \eqref{gammabeta}. Since the transformations in \eqref{iptorhp}, \eqref{ytot}, \eqref{ttos} and \eqref{stor} are all invertible, we trace back the transformations $\mathbf{P} \to \mathbf{Y} \to \mathbf{T} \to \mathbf{S} \to \mathbf{R}$ and obtain
\begin{equation}\label{outside}
    \mathbf{P}(z) = K^{-1} e^{\frac{nl}{2} \sigma_3} \mathbf{R}(z) \mathbf{N}(z) e^{n(g(z) - \frac{l}{2}) \sigma_3} K,
\end{equation}
for $z$ bounded away from the positive real line and
\begin{equation}\label{V}
        \mathbf{P}(z) = \begin{cases}
            K^{-1} e^{\frac{nl}{2} \sigma_3} \mathbf{R}(z) \mathbf{N}(z) e^{n(g(z) - \frac{l}{2}) \sigma_3} K D_{+}^u(z) ^{-1}, \quad \textrm{for $\im{z} > 0$,}\\
            K^{-1} e^{\frac{nl}{2} \sigma_3} \mathbf{R}(z) \mathbf{N}(z) e^{n(g(z) - \frac{l}{2}) \sigma_3} K D_{-}^u(z) ^{-1}, \quad \textrm{for $\im{z} < 0$,}
        \end{cases}
    \end{equation}
for $z$ close to the positive real line but bounded away form the support of the equilibrium measure $(0, b)$. Note that the functions $D_{\pm}^u(z)$ are exponentially close to the identity matrix as $z \to \infty$ from their definitions in \eqref{Du}.  And obviously, it follows from \eqref{R-large-z} and \eqref{Rbehaviour} that
\begin{equation}
    \mathbf{R}_1 = O\left( \frac{1}{n} \right), \quad \mathbf{R}_2 = O\left( \frac{1}{n} \right).
\end{equation}
From the definition of the $g$-function in \eqref{g}, we have
\begin{equation}\label{gforzlarge}
    e^{n g(z) \sigma_3}\begin{pmatrix}
        z^{-n} & 0\\
        0 & z^{n}
    \end{pmatrix} = I + \frac{G_1}{z} + \frac{G_2}{z^2} + O\left(\frac{1}{z^3}\right), \qquad \textrm{as } z \to \infty,
\end{equation}
with
    $[G_1]_{12} = [G_1]_{21} = 0$ and $[G_1]_{11} + [G_1]_{22} = 0. $ Recalling the large-$z$ behaviour of $\mathbf{N}(z) $ in \eqref{Mforzlarge}, we have from the above formulas
    \begin{equation}\label{P1}
        \mathbf{P}_1 = K^{-1} e^{\frac{nl}{2}\sigma_3} \Big[\mathbf{R}_1+\mathbf{N}_1+G_1 \Big]e^{-\frac{nl}{2}\sigma_3}K
    \end{equation}
    and
    \begin{equation}\label{P2}
        \mathbf{P}_2 = K^{-1} e^{\frac{nl}{2}\sigma_3} \Big[\mathbf{R}_2+\mathbf{N}_2+G_2 + \mathbf{R}_1\mathbf{N}_1 + \mathbf{R}_1 G_1+\mathbf{N}_1 G_1 \Big]e^{-\frac{nl}{2}\sigma_3}K,
    \end{equation}
 where $\mathbf{P}_1$ and $\mathbf{P}_2$ are the coefficients in \eqref{Pinfty}, $\mathbf{N}_1$ and $\mathbf{N}_2$ are given in \eqref{N1-def} and \eqref{N2-def}.

 Finally, we get the asymptotics of $h_{n,n}$ in  \eqref{hn-final} by  inserting \eqref{P1} into \eqref{hN}.  Similarly, the asymptotics of the recurrence coefficients $\mathscr{A}_{n, n}$ and $\mathscr{B}_{n, n}$ in \eqref{an-final} and \eqref{bn-final} are obtained by substituting  \eqref{P1} and \eqref{P2} into \eqref{gammabeta}.

This completes the proof of Theorem \ref{coefficient}. \hfill $\Box$

%


\subsection{Proofs of Theorem \ref{asymptoticinv} to \ref{asymptotica}}

Next, we derive the asymptotics of orthogonal polynomials $P_{n, n}(z)$ in different regions in the complex plane as $n \to \infty$. We first consider the outside region, which is bounded away from $(0, b)$.
\begin{proof}[Proof of Theorem \ref{asymptoticinv}]
    For $z$ bounded away form the support of the equilibrium measure, we use the formulas for $\mathbf{P}(z)$ given in \eqref{outside} and \eqref{V}. Since the first column of $D_{\pm}^u(z)$ is $\begin{pmatrix} 1 \\ 0 \end{pmatrix}$, the formulas \eqref{outside} and \eqref{V} admit the same (1,1)-entry, which gives us
    \begin{equation}
        P_{n, n}(z) = [\mathbf{P}(z)]_{11} = e^{ng(z)}\left([\mathbf{R}(z)]_{11} [\mathbf{N}(z)]_{11}+[\mathbf{R}(z)]_{12} [\mathbf{N}(z)]_{21} \right).
    \end{equation}
    With the explicit expression of $\mathbf{N}(z)$ in \eqref{global} and the asymptotics of $\mathbf{R}(z)$ in \eqref{Rbehaviour}, we obtain \eqref{pn-outside}.

    This finishes the proof of Theorem \ref{asymptoticinv}.
\end{proof}
For the oscillating region, we divide the interval $(0, b)$ into two subintervals: the band $(a, b)$ and the saturated region $(0, a)$.
\begin{proof}[Proof of Theorem \ref{asymptoticinb}]
    For $z$ close to the interval $(a, b)$ and bounded away from the endpoints $a$ and $b$, from the invertible transformations \eqref{iptorhp}, \eqref{ytot}, \eqref{ttos} and \eqref{stor}, we get
    \begin{equation}\label{B}
        \mathbf{P}(z) = \begin{cases}
            K^{-1} e^{\frac{nl}{2} \sigma_3} \mathbf{R}(z) \mathbf{N}(z) B(z) e^{n(g(z) - \frac{l}{2}) \sigma_3} K D_{+}^u(z) ^{-1}, \quad \textrm{for $\im{z} > 0$,}\\
            K^{-1} e^{\frac{nl}{2} \sigma_3} \mathbf{R}(z) \mathbf{N}(z) B(z)^{- 1} e^{n(g(z) - \frac{l}{2}) \sigma_3} K D_{-}^u(z) ^{-1}, \quad \textrm{for $\im{z} < 0$.}
        \end{cases}
    \end{equation}
    Taking limit as $z$ approaches the real line from the upper half plane, we have the asymptotics of its (1,1)-entry
    \begin{equation}\label{B1}
        P_{n, n}(x) = [\mathbf{P}(x)]_{11} = \left([\mathbf{N}(x)]_{11,+}e^{ng_{+}(x)} + x^{-\alpha+\frac{1}{2}}[\mathbf{N}(x)]_{12,+}e^{-n(g_{+}(x)-cx-l)}  \right)\left(1+O\left( \frac{1}{n} \right)\right).
    \end{equation}
    Recalling the function $\mathbf{N}(z)$ in \eqref{global} and using the property of $g$-function in \eqref{g+_-2}, we obtain
    \begin{align}
        P_{n, n}(x) = & \frac{D_{\infty}e^{\frac{n}{2}(c x + l)}}{2 \sqrt{x} \left((\sqrt{a} + \sqrt{b})x\right)^{\alpha - \frac{1}{2}}(x-a)^{\frac{1}{4}}(b-x)^{\frac{1}{4}}} \nonumber\\
            & \times \left[ \left(x + \sqrt{a b} + i \sqrt{(x-a)(b-x)}\right)^{\alpha - \frac{1}{2}} (x + i \sqrt{(x-a)(b-x)}) e^{i n \pi \int_x^{b} \rho (s) ds - \frac{\pi i}{4}} \right. \nonumber\\
            & + \left(x + \sqrt{a b} - i \sqrt{(x-a)(b-x)}\right)^{\alpha - \frac{1}{2}} (x - i \sqrt{(x-a)(b-x)}) e^{-i n \pi \int_x^{b} \rho (s) ds + \frac{\pi i}{4}} \nonumber\\
            & \left. + O\left( \frac{1}{n} \right)\right]. \label{B11}
    \end{align}
    To put the above formula into a more concise form, let us rewrite the following quantities in exponential form:
    \begin{equation}\label{exp-1}
        x + \sqrt{a b} \pm i \sqrt{(x-a)(b-x)} = (\sqrt{a} + \sqrt{b})\sqrt{x} \exp{\left(\pm i \arccos{\frac{x + \sqrt{a b}}{(\sqrt{a} + \sqrt{b})\sqrt{x}}}\right)}
    \end{equation}
    and
    \begin{equation}\label{exp-2}
        x \pm i \sqrt{(x-a)(b-x)} = \sqrt{(a+b)x - a b} \exp{\left(\pm i \arccos{\frac{x}{\sqrt{(a+b)x - a b}}}\right)}.
    \end{equation}
    Inserting \eqref{exp-1} and \eqref{exp-2} into \eqref{B11} gives us
    \begin{equation}\label{B2}
        \begin{split}
            P_{n, n}(x) = & \frac{D_{\infty}e^{\frac{n}{2}(cx + l)}\sqrt{(a+b)x-a b}}{x^{\frac{\alpha}{2}}(x-a)^{\frac{1}{4}}(b-x)^{\frac{1}{4}}}\\
            & \times \left[e^{i \left((\alpha - \frac{1}{2})\arccos{\frac{x + \sqrt{a b}}{(\sqrt{a} + \sqrt{b})\sqrt{x}}} + \arccos{\frac{x}{\sqrt{(a+b)x - a b}}} + n \pi \int_x^{b} \rho (s) ds - \frac{\pi i}{4} \right)} \right.\\
            & \left.+ e^{-i \left((\alpha - \frac{1}{2})\arccos{\frac{x + \sqrt{a b}}{(\sqrt{a} + \sqrt{b})\sqrt{x}}} + \arccos{\frac{x}{\sqrt{(a+b)x - a b}}} + n \pi \int_x^{b} \rho (s) ds - \frac{\pi i}{4} \right)} + O\left( \frac{1}{n} \right) \right].
        \end{split}
    \end{equation}
    Then the approximation \eqref{PninB-final} follows from the above formula. A similar calculation leads to \eqref{PninB-final} as $z$ approaches the real line from the lower half plane.

    This completes the proof of Theorem \ref{asymptoticinb}.
\end{proof}
Next, we derive the asymptotic behaviour of $P_{n, n}(x)$ for $x$ in a compact subset of the saturated region.
\begin{proof}[Proof of Theorem \ref{asymptoticins}]
For $z$ close to the interval $(0, a)$ and bounded away from the endpoints $0$ and $a$, the transformations \eqref{iptorhp}, \eqref{ytot}, \eqref{ttos} and \eqref{stor} give us
\begin{equation}\label{S}
    \mathbf{P}(z) = \begin{cases}
        K^{-1} e^{\frac{n l}{2} \sigma_3} \mathbf{R}(z) \mathbf{N}(z) A_{+}(z)^{-1} e^{n(g(z) - \frac{l}{2}) \sigma_3} K D_{+}^l(z) ^{-1}, \quad \textrm{for $\im{z} > 0$,}\\
        K^{-1} e^{\frac{n l}{2} \sigma_3} \mathbf{R}(z) \mathbf{N}(z) A_{-}(z)^{-1} e^{n(g(z) - \frac{l}{2}) \sigma_3} K D_{-}^l(z) ^{-1}, \quad \textrm{for $\im{z} < 0$.}
    \end{cases}
\end{equation}
We substitute the matrices $A_{\pm}(z)$ in \eqref{A-def} and $D_{\pm}^l(z)$ in \eqref{Dl} into the above formula and take limit as $z$ approaches the real line from the upper half plane. Applying the asymptotics of $\mathbf{R}(z)$ in \eqref{Rbehaviour} and using the properties of $g$-function in \eqref{2g-V-l} and \eqref{g+_-3}, we get the approximation of its (1,1)-entry
    \begin{align}
        P_{n, n}(x)  = & [\mathbf{P}(x)]_{11} = -2i \sin{(n \pi \sqrt{x})}\left[e^{i n \pi \sqrt{x} + n g_{+}(x)}[\mathbf{N}(x)]_{11,+}\left(1+O\left( \frac{1}{n} \right)\right)\right. \nonumber\\
        &\left.+ e^{-n(g_{+}(x)-cx-l)} [\mathbf{N}(x)]_{12,+}x^{-\alpha+\frac{1}{2}} \left(1+O\left( \frac{1}{n} \right)\right)\right]\nonumber\\
        = & (-1)^n e^{n \int_0^{b} \log |x-s| \rho (s) ds}\left[-2i \sin{(n \pi \sqrt{x})}[\mathbf{N}(x)]_{11,+}\left(1+O\left( \frac{1}{n} \right)\right) + O(e^{-n \delta})\right].
    \end{align}
With the explicit expression of $\mathbf{N}(z)$ in \eqref{global} we obtain formula \eqref{PninS-final}. A similar calculation leads to the same formula \eqref{PninS-final} as $z$ approaches the real line from the lower half plane.

This finishes the proof of Theorem \ref{asymptoticins}.
\end{proof}

Similarly, for orthogonal polynomials $P_{n, n}(z)$ at the endpoints $0$, $a$ and $b$, the asymptotics can be obtained in terms of explicit transformations near these endpoints.
We first consider the asymptotic behaviour of $P_{n, n}(z)$ near the origin.
\begin{proof}[Proof of Theorem \ref{asymptotic0}]
    A combination of \eqref{iptorhp}, \eqref{ytot}, \eqref{ttos} and \eqref{stor} gives
    \begin{equation}\label{Pin0-}
        \mathbf{P}(z) =
        K^{-1} e^{\frac{n l}{2} \sigma_3} \mathbf{R}(z) \mathbf{P}^{(0)}(z) e^{n(g(z) - \frac{l}{2}) \sigma_3} K,
    \end{equation}
    for $z \in D(0, \varepsilon)$ with $\re{z} < 0$ and
    \begin{equation}\label{Pin0+}
        \mathbf{P}(z) = \begin{cases}
        K^{-1} e^{\frac{n l}{2} \sigma_3} \mathbf{R}(z) \mathbf{P}^{(0)}(z) A_{+}(z)^{-1} e^{n(g(z) - \frac{l}{2}) \sigma_3} K D_{+}^l(z) ^{-1}, \quad \textrm{for $\im{z} > 0$,}\\
         K^{-1} e^{\frac{n l}{2} \sigma_3} \mathbf{R}(z) \mathbf{P}^{(0)}(z) A_{-}(z)^{-1} e^{n(g(z) - \frac{l}{2}) \sigma_3} K D_{-}^l(z) ^{-1}, \quad \textrm{for $\im{z} < 0$,}
        \end{cases}
    \end{equation}
    for $z \in D(0, \varepsilon)$ with $\re{z} > 0$.

    Substituting the definition of the parametrix $\mathbf{P}^{(0)}(z)$ in \eqref{P0-def} into \eqref{Pin0-} and using the asymptotics of $\mathbf{R}(z)$ in \eqref{Rbehaviour}, we have
    \begin{equation}
        \mathbf{P}(z) =
        K^{-1} e^{\frac{n l}{2} \sigma_3} \mathbf{N}(z) H^{*}(z)^{\sigma_3} e^{n(g(z) - \frac{l}{2}) \sigma_3} K \left(I + O\left( \frac{1}{n} \right) \right).
    \end{equation}
    Expanding the above formula and taking its (1,1)-entry, the asymptotic approximation \eqref{Pnn-void-main} of $P_{n, n}(z)$ for $\re{z}<0$ is obtained.
    By using \eqref{H*-def}, one can get the same formula \eqref{Pnn-void-main} for $\re{z}>0$.

    This proves Theorem \ref{asymptotic0}.
\end{proof}
Then, we consider the asymptotics of $P_{n, n}(z)$ near the band-void edge point $b$.
\begin{proof}[Proof of Theorem \ref{asymptoticb}]
    For $z \in D(b, \varepsilon)$, we follow the explicit transformations \eqref{iptorhp}, \eqref{ytot}, \eqref{ttos} and \eqref{stor} to obtain
    \begin{equation}\label{Pinb+}
        \mathbf{P}(z) =\begin{cases}
            K^{-1} e^{\frac{n l}{2} \sigma_3} \mathbf{R}(z) \mathbf{P}^{(b)}(z) e^{n(g(z) - \frac{l}{2}) \sigma_3} K D_{+}^u(z) ^{-1}, \quad & \textrm{for $\im{z}>0$}\\
            K^{-1} e^{\frac{n l}{2} \sigma_3} \mathbf{R}(z) \mathbf{P}^{(b)}(z) e^{n(g(z) - \frac{l}{2}) \sigma_3} K D_{-}^u(z) ^{-1}, \quad & \textrm{for $\im{z}<0$}
        \end{cases}
    \end{equation}
    with $\re{z} > b$ and
    \begin{equation}\label{Pinb-}
       \mathbf{P}(z) = \begin{cases}
           K^{-1} e^{\frac{n l}{2} \sigma_3} \mathbf{R}(z) \mathbf{P}^{(b)}(z) B(z) e^{n(g(z) - \frac{l}{2}) \sigma_3} K D_{+}^u(z) ^{-1}, \quad & \textrm{for $\im{z}>0$}\\
            K^{-1} e^{\frac{n l}{2} \sigma_3} \mathbf{R}(z) \mathbf{P}^{(b)}(z) B(z)^{- 1} e^{n(g(z) - \frac{l}{2}) \sigma_3} K D_{-}^u(z) ^{-1}, \quad & \textrm{for $\im{z}<0$}
       \end{cases}
   \end{equation}
   with $\re{z} < b$.
    Applying the asymptotics of $\mathbf{R}(z)$ in \eqref{Rbehaviour} and recalling the definition of the parametrix $\mathbf{P}^{(b)}(z)$ in \eqref{Pb-def}, we obtain
    \begin{equation}
        \begin{split}
            \mathbf{P}(z) = &  \sqrt{\pi} e^{-\frac{\pi i}{12}} K^{-1} e^{\frac{n l}{2} \sigma_3} \mathbf{N}(z) z^{(\frac{\alpha}{2}- \frac{1}{4}) \sigma_3} e^{\frac{\pi i}{4} \sigma_3} \begin{pmatrix}
            1 & -1\\
            1 & 1
        \end{pmatrix} \\
        & \times n^{\frac{1}{6} \sigma_3} f(z)^{\frac{1}{4} \sigma_3}
          \begin{pmatrix}
            \Ai{(n^{\frac{2}{3}}f(z))} & \Ai{(\omega^2 n^{\frac{2}{3}}f(z))}\\
            \Ai'{(n^{\frac{2}{3}}f(z))} & \omega^2 \Ai'{(\omega^2 n^{\frac{2}{3}}f(z))}
        \end{pmatrix}\\
        & \times e^{-\frac{\pi i}{6} \sigma_3} e^{n \frac{cz}{2} \sigma_3} z^{(-\frac{\alpha}{2}+ \frac{1}{4}) \sigma_3} e^{n(g(z) - \frac{l}{2}) \sigma_3} K D_{+}^u(z)^{-1} \left(I + O\left( \frac{1}{n} \right) \right)
        \end{split}
    \end{equation}
    for $\re{z}>b$ and $\im{z}>0$. After a straightforward computation, we get the approximation \eqref{Pnn-b-main} from the (1,1)-entry of the above formula,. In a similarly way, one can show that \eqref{Pnn-b-main} also holds for $z \in \{\re{z} > b, \im{z} < 0\} \cup \{\re{z}<b\}$.

    This completes the proof of Theorem \ref{asymptoticb}.
\end{proof}
Finally, we derive the asymptotic behaviour of the polynomials $P_{n, n}(z)$ near the saturated region-band edge point $a$.
\begin{proof}[Proof of Theorem \ref{asymptotica}]
     For $z \in D(a, \varepsilon)$, we follow the explicit transformations \eqref{iptorhp}, \eqref{ytot}, \eqref{ttos} and \eqref{stor} to obtain
      \begin{equation}\label{Pina+}
       \mathbf{P}(z) = \begin{cases}
           K^{-1} e^{\frac{n l}{2} \sigma_3} \mathbf{R}(z) \mathbf{P}^{(a)}(z) B(z) e^{n(g(z) - \frac{l}{2}) \sigma_3} K D_{+}^u(z) ^{-1}, \quad & \textrm{for $\im{z}>0$}\\
            K^{-1} e^{\frac{n l}{2} \sigma_3} \mathbf{R}(z) \mathbf{P}^{(a)}(z) B(z)^{- 1} e^{n(g(z) - \frac{l}{2}) \sigma_3} K D_{-}^u(z) ^{-1}, \quad & \textrm{for $\im{z}<0$}
       \end{cases}
   \end{equation}
   with $\re{z} > a$ and
   \begin{equation}\label{Pina-}
        \mathbf{P}(z) = \begin{cases}
        K^{-1} e^{\frac{n l}{2} \sigma_3} \mathbf{R}(z) \mathbf{P}^{(a)}(z) A_{+}(z)^{-1} e^{n(g(z) - \frac{l}{2}) \sigma_3} K D_{+}^l(z) ^{-1}, \quad \textrm{for $\im{z} > 0$}\\
         K^{-1} e^{\frac{n l}{2} \sigma_3} \mathbf{R}(z) \mathbf{P}^{(a)}(z) A_{-}(z)^{-1} e^{n(g(z) - \frac{l}{2}) \sigma_3} K D_{-}^l(z) ^{-1}, \quad \textrm{for $\im{z} < 0$}
        \end{cases}
    \end{equation}
    with $\re{z} < a$. Next, we apply the asymptotics of $\mathbf{R}(z)$ in \eqref{Rbehaviour} and the definition of the parametrix $\mathbf{P}^{(a)}(z)$ in \eqref{Pa-def} to get
    \begin{equation}
        \begin{split}
            \mathbf{P}(z) = & (-1)^n  \sqrt{\pi} e^{-\frac{\pi i}{12}} K^{-1} e^{\frac{n l}{2} \sigma_3} \mathbf{N}(z) z^{(\frac{\alpha}{2}- \frac{1}{4}) \sigma_3} e^{\frac{\pi i}{4} \sigma_3} \begin{pmatrix}
            1 & -1\\
            1 & 1
        \end{pmatrix} n^{\frac{1}{6} \sigma_3} (-\widetilde{f}(z))^{\frac{1}{4} \sigma_3} \\
        & \times \sigma_3 \begin{pmatrix}
            \Ai{(-n^{\frac{2}{3}}\widetilde{f}(z))} & -\omega^2 \Ai{(-\omega n^{\frac{2}{3}}\widetilde{f}(z))}\\
            \Ai'{(-n^{\frac{2}{3}}\widetilde{f}(z))} &  -\Ai'{(-\omega n^{\frac{2}{3}}\widetilde{f}(z))}
        \end{pmatrix} e^{-\frac{\pi i}{6} \sigma_3} \begin{pmatrix}
            1 & 0\\
            1 & 1
        \end{pmatrix} \sigma_3 \sigma_1\\
        & \times  e^{-n(g(z) - \frac{cz}{2}- \frac{l}{2} + i \pi \sqrt{z}) \sigma_3} z^{(-\frac{\alpha}{2}+ \frac{1}{4}) \sigma_3} B(z)e^{n(g(z) - \frac{l}{2}) \sigma_3} K D_{+}^u(z)^{-1} \left(I + O\left( \frac{1}{n} \right) \right)
        \end{split}
    \end{equation}
    for $\re{z}>a$ and $\im{z}>0$. Then, the asymptotics of $P_{n, n}(z)$ in \eqref{Pnn-a-main} is obtained by substituting the formula \eqref{J-def} for $B(z)$ into the above formula. In a similarly way, we get the asymptotic expansion \eqref{Pnn-a-main} for $z \in \{\re{z}>a, \im{z}<0\}\cup \{\re{z}<a\}$.

    This finishes the proof of Theorem \ref{asymptotica}.
\end{proof}

\section{Discussion}
\label{sec:dl:discussion}

As we have mentioned in the Introduction, the upper constraint in the equilibrium problem \eqref{em-def} is not active  when $c < \frac{\pi^2}{4}$. Then, the equilibrium measures $\mu_0$ are the same for both the discrete and continuous Laguerre polynomials. It is natural to expect that these polynomials have the same asymptotic expansions as well. Note that,  0 is the hard edge of $\mu_0$ when $c < \frac{\pi^2}{4}$; see the density function in \eqref{em:rho0}. In the steepest descent analysis of the RH analysis for continuous Laguerre polynomials, the local parametrix is constructed in terms of the Bessel functions; see \cite{vanlessen2007}. However, when performing the RH analysis for the discrete Laguerre polynomials, we failed to arrive at the model Bessel parametrix near 0. Moreover, some numerical computations also suggest that there exist certain differences for the zeros of the discrete and continuous Laguerre polynomials near the origin. This is an unexpected and interesting observation.


\subsection{Local parametrices near $0$}

Let us first elaborate the difference from the viewpoint of the RH analysis.
For the discrete case, the interpolation problem remains the same as that in Section \ref{sec:dl:ip}. However, when we transform it to a continuous RH problem, the regions $\Omega _{\pm}^{\bigtriangleup}$ related to the saturated region disappear; see Figure \ref{sigmaY}. Then, the transformation in \eqref{iptorhp} is simplified to be
\begin{equation} \label{iptorhp2}
    \mathbf{Y}(z) = \begin{cases}
            K\mathbf{Y}^u(z)K^{-1}, & \textrm{$z \in \Omega _{\pm}^{\bigtriangledown}$},\\
            K\mathbf{P}(z)K^{-1}, & \textrm{otherwise},
            \end{cases}
\end{equation}
where $\mathbf{Y}^u$ and $K$ are given in \eqref{yu-def} and \eqref{K}. As a consequence, following the similar analysis as in Section \ref{sec:dl:rhp},  the local parametrix near the origin is modified as follows. (Let us denote  the local parametrices for the discrete and continuous Laguerre polynomials near the origin by $\mathbf{P}^{(d,0)}(z)$ and $\mathbf{P}^{(c,0)}(z)$, respectively.)

\subsubsection*{\emph{Local parametrix near $0$ for the discrete case: $\mathbf{P}^{(d,0)}(z)$}}
\begin{figure}[h]
    \centering
    \includegraphics[width=0.3\textwidth]{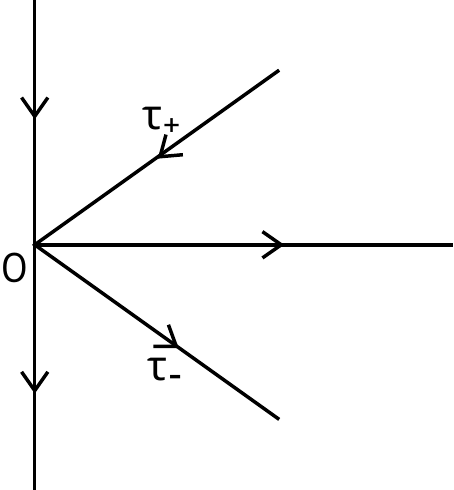}
    \caption{Contour $\Sigma^{(d)}_0$.}
    \label{sigma0}
\end{figure}
\begin{itemize}
    \item {\emph{Analyticity}: $\mathbf{P}^{(d,0)}(z)$ is analytic for $z \in D(0, \varepsilon) \setminus \Sigma^{(d)}_0$.}
    \item{\emph{Jump condition}: $\mathbf{P}^{(d,0)}_{+}(z)=\mathbf{P}^{(d,0)}_{-}(z)J_{\mathbf{P}^{(d,0)}}(z)$ for $z \in D(0, \varepsilon) \cap \Sigma^{(d)}_0$,
    \begin{equation} \label{dL-jump-sec6}
        J_{\mathbf{P}^{(d, 0)}}(z)=
        \begin{cases}
            \begin{pmatrix}
                0 & x^{\alpha-\frac{1}{2}}\\
                -x^{-\alpha+\frac{1}{2}} & 0
            \end{pmatrix}, & z \in (0, \varepsilon) ,\\
            \begin{pmatrix}
                1 & \pm z^{\alpha-\frac{1}{2}} \frac{e^{n(2g(z)-V(z)-l)}}{1- e ^{\mp 2 i n \pi \sqrt{z}}}\\
                0 & 1
            \end{pmatrix}, & z \in (0, \pm i \varepsilon),\\
            \begin{pmatrix}
                1 & 0\\
                \mp z^{-\alpha+\frac{1}{2}}e^{-n(2g(z)-V(z)-l)} & 1
            \end{pmatrix}, & z \in \tau_{\pm}.
        \end{cases}
    \end{equation}
    }
      \item{\emph{Matching condition}: for $z \in \partial D(0, \varepsilon)$,
    \begin{equation}
        \mathbf{P}^{(d,0)}(z) = \left(I+ O\left( \frac{1}{n} \right)\right) \mathbf{N}(z), \qquad \textrm{as } n \to \infty.
    \end{equation}}
\end{itemize}

While for the continuous Laguerre polynomials, since the first ``discrete-to-continuous" transformation \eqref{iptorhp2} is not needed, the local parametrix near the origin is given by (see also Vanlessen \cite{vanlessen2007}):

\subsubsection*{\emph{Local parametrix near $0$ for the continuous case: $\mathbf{P}^{(c,0)}(z)$}}
\begin{figure}[h]
    \centering
    \includegraphics[width=0.3\textwidth]{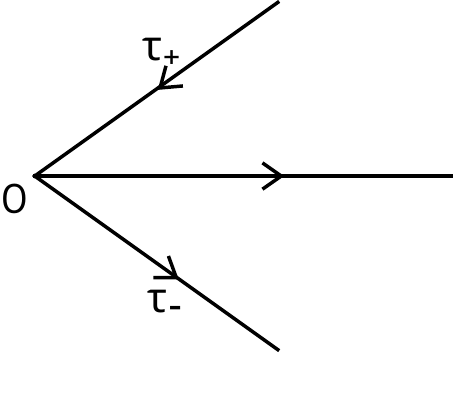}
    \caption{Contour $\Sigma^{(c)}_0$.}
    \label{sigma0c}
\end{figure}
\begin{itemize}
    \item {\emph{Analyticity}: $\mathbf{P}^{(c,0)}(z)$ is analytic for $z \in D(0, \varepsilon) \setminus \Sigma^{(c)}_0$.}
    \item{\emph{Jump condition}: $\mathbf{P}^{(c,0)}_{+}(z)=\mathbf{P}^{(c,0)}_{-}(z)J_{\mathbf{P}^{(c,0)}}(z)$ for $z \in D(0, \varepsilon) \cap \Sigma^{(c)}_0$,
    \begin{equation}
        J_{\mathbf{P}^{(c, 0)}}(z)=
        \begin{cases}
            \begin{pmatrix}
                0 & x^{\alpha-\frac{1}{2}}\\
                -x^{-\alpha+\frac{1}{2}} & 0
            \end{pmatrix}, & z \in (0, \varepsilon), \\
            \begin{pmatrix}
                1 & 0\\
                \mp z^{-\alpha+\frac{1}{2}}e^{-n(2g(z)-V(z)-l)} & 1
            \end{pmatrix}, & z \in \tau_{\pm}.
        \end{cases}
    \end{equation}
    }
     \item{\emph{Matching condition}: for $z \in \partial D(0, \varepsilon)$,
    \begin{equation}
        \mathbf{P}^{(c,0)}(z) = \left(I+ O\left( \frac{1}{n} \right)\right) \mathbf{N}(z), \qquad \textrm{as } n \to \infty.
    \end{equation}}
\end{itemize}
Comparing with the above RH problem, there is an extra jump on the imaginary axis in \eqref{dL-jump-sec6} for the RH problem for $\mathbf{P}^{(d,0)}(z)$, such that the model Bessel parametrix cannot be applied to approximate $\mathbf{P}^{(d,0)}(z)$. Note that off-diagonal entry $\pm z^{\alpha-\frac{1}{2}} \frac{e^{n(2g(z)-V(z)-l)}}{1- e ^{\mp 2 i n \pi \sqrt{z}}}$ on the imaginary axis in \eqref{dL-jump-sec6} does not tend to 0 uniformly in the neighbourhood of the origin. Therefore it has a contribution on the local parametrix. For the case $c > \frac{\pi^2}{4}$ we have studied, there is also a jump on the imaginary axis in \eqref{Y-jump}, such that the local parametrix is constructed in terms of the Gamma functions in \eqref{P0-def}. However, we have not found a suitable parametrix near $0$ when $c<\frac{\pi^2}{4}$.

\subsection{Numerical computation for the zeros}

One can also see some numerical evidences between these two cases. Let us denote the discrete and continuous Laguerre polynomials by $\dL_{n}^{(0)}(x)$ and $\cL_{n}^{(-/2)}(x)$, which are orthogonal with respect to the weight functions $e^{-\frac{\pi^2}{60}x}$ and $x^{-\frac{1}{2}}e^{-\frac{\pi^2}{60}x}$, respectively. The reason why there is $ x^{-\frac{1}{2}}$ difference in the weight function is due to the ``discrete-to-continuous" transformation \eqref{iptorhp2}; see also Remark \ref{changeweight}.


\begin{table}[h]
    \centering
    \begin{tabular}{|c|c|}
        \hline
        $\dL_{10}^{(0)}(x)$    & $\cL_{10}^{(-1/2)}(x)$\\
        \hline
        1.0312593902618079872  & 0.3659238650514353556 \\
        4.3927946544706143130  & 3.3063179324671812008 \\
        10.508882642411045983  & 9.2583899628419669141 \\
        19.696609776199878331  & 18.374677972612339445 \\
        32.270084609499775568  & 30.912532317124747650 \\
        48.658595847104970581  & 47.281160918670718658 \\
        69.526822240006454052  & 68.137261123322616690 \\
        95.99803545442174504   & 94.600529260150193532 \\
        130.24647289010848939  & 128.84341399111556911 \\
        177.85779728345868061  & 176.45053941796852867 \\
        \hline
    \end{tabular}
    \caption{Zeros of discrete and continuous Laguerre polynomials}
    \label{table:laguerre}
\end{table}
The above Table gives 10 zeros for both polynomials. When $z$ is large, the zeros are close to each other; while when $z$ is close to the origin, there are some obvious differences between these two cases.

\section*{Acknowledgements}

This work was partially supported by a grant from the City University of Hong Kong (Project No. 7005252), and a grant from the Research Grants Council of the Hong Kong Special Administrative Region, China (Project No. CityU 11300520). We would like to thank Shuai-Xia Xu for helpful discussions.

\end{document}